\documentclass[11pt,reqno]{siamart190516}

\usepackage{amsfonts,amsmath}
\usepackage{amssymb}
\usepackage{amsfonts}
\usepackage{bm,bbm}

\usepackage{algorithm}
\usepackage[noend]{algpseudocode} 

\usepackage{xfrac}
\usepackage{caption}
\usepackage{capt-of} 

\usepackage{cite}
\usepackage{cleveref}

\usepackage{color}
\usepackage{hyperref}
\usepackage{graphicx,epstopdf} 
\usepackage[caption=false]{subfig} 
\graphicspath{{./Figures/}}
\usepackage{booktabs} 

\usepackage{comment}

 \usepackage{tikz}
 \usetikzlibrary{backgrounds,mindmap}
\usetikzlibrary{shapes.geometric, arrows}

\newcommand{\abs}[1]{\left\vert#1\right\vert}

\newcommand{\norm}[1]{\left\vert \left\vert #1 \right\vert\right\vert}

\newcommand{\Real}{\mathbb{R}}

\DeclareMathOperator\tr{Tr}

\newtheorem{thm}{Theorem}[section]

\newtheorem{lem}[thm]{Lemma}

\newtheorem{rem}[thm]{Remark}
\newtheorem{exm}[thm]{Example}

\title{State-dependent temperature control in Langevin diffusions using  numerical exploratory Hamiltonian-Jacobi-Bellman equations}
\author{Taorui Wang\thanks{Department of Mathematical Sciences, Worcester Polytechnic Institute, 
Worcester, MA. USA. (\email{twang13@wpi.edu}, 
\email{gwang2@wpi.edu}, 
\email{zzhang7@wpi.edu}) Gu Wang is partially supported
by National Science Foundation grant \#DMS - 2206282.}
\and 
Xun Li\thanks{Department of Applied Mathematics,
Hong Kong Polytechnic University,
  Hong Kong. China. \email{(li.xun@polyu.edu.hk}) Xun Li is supported by the Research Grants Council of Hong Kong under grants 15225124, PolyU 4-ZZVB, 4-ZZP4. }
\and 
Gu Wang\footnotemark[1]
  \and 
Zhongqiang Zhang\footnotemark[1]  
  }
\date{\today}

\begin{document}
\maketitle
\begin{abstract}
 Choosing how much noise to add in Langevin dynamics is essential for making these algorithms effective in challenging optimization problems.
One promising approach is to determine this noise by solving Hamilton–Jacobi–Bellman (HJB) equations and their exploratory variants. Though these ideas have been demonstrated to work well in one dimension, extension to high-dimensional minimization has been limited by two unresolved numerical challenges: setting reliable control bounds and stably computing the second‑order information (Hessians) required by the equations.  
 These issues and the broader impact of HJB parameters have not been systematically examined. This work provides the first such investigation. We introduce principled control bounds and develop a physics‑informed neural network framework that embeds the structure of exploratory HJB equations directly into training, stabilizing computation, and enabling accurate estimation of state‑dependent noise in high‑dimensional problems. Numerical experiments demonstrate that the resulting method remains robust and effective well beyond low‑dimensional test cases.
\end{abstract}

 \textbf{keywords}    
 stochastic optimization, exploratory control, entropy‑regularized HJB equations, gradient descent methods, scientific machine learning, high‑dimensional PDEs
 


\section{Introduction}%
In gradient descent methods for optimization, one common strategy to escape local minima and saddle points is the addition of noise, leading to Langevin dynamics. 
Designing such noise terms, however, is nontrivial and often problem-dependent.
From a theoretical perspective, noise can be controlled through diffusion terms, with the optimal control derived from the Hamilton–Jacobi–Bellman (HJB) equation (see the workflow below). 
Yet, this classical approach often yields extreme (bang--bang) noise levels, which can be rigid and numerically brittle \cite{5445043, MR1923277}.
 To address this limitation, Gao et al. \cite{GaoXuZhou2022} proposed a randomized control with entropy regularization. Their method, termed the exploratory HJB equation, demonstrated clear benefits in one-dimensional problems, such as escaping local minima of a double-well function. Tang et al. \cite{TangZhangZhou2022} further established that the exploratory HJB equation converges to the classical HJB equation as the regularization vanishes.
\begin{center}
\tikzstyle{startstop} = [rectangle, rounded corners, minimum width=3cm, minimum height=1cm,text centered, draw=black, fill=white!30]
\tikzstyle{arrow} = [thick,->,>=stealth]

\begin{tikzpicture}[node distance=4cm, scale=0.4]

\node (problem) [startstop] {$\min_{x\in \mathbb{R}^d} f(\bm{x})$};
\node (method) [startstop, right of = problem] {Langevin dynamics};
\node (approach) [startstop, right of = method, xshift=1cm] 
    {\shortstack{Hamilton--Jacobi--Bellman\\ (HJB) equation}};

\node (solver) [startstop, below of = approach, xshift=-1.5cm, yshift=2.6cm] 
    {regularization, parameters, solvers};

\node (control) [startstop, left of = solver, xshift=-2cm] 
    {design control (temperature)};

\draw [arrow] (method) -- (problem);
\draw [thick,solid] (solver) -- (approach);
\draw [arrow] (control) -- (method);
\draw [arrow] (approach) -- (control);

\end{tikzpicture}
\end{center}


Despite these advances, the extent to which exploratory HJB formulations provide reliable state-dependent temperature schedules in higher dimensions remains insufficiently understood, particularly in landscapes with numerous saddle points and competing local minima. 
Furthermore, scalable solvers for both classical and exploratory HJB equations are still under active development. Classical discretization-based approaches have been investigated in low dimensions, e.g., in \cite{FengLewisRapp2022}, while deep learning methods have been explored in low and high dimensions, e.g., in \cite{LiuEtAl2023PINNHJB,KimChoKimKim2025,MengEtAl2024PINNPI,dupret2026deep}. The underlying computational difficulty stems from strong nonlinearity combined with the need to resolve solutions on high-dimensional domains.

In this work, we \textbf{develop practical numerical methods for solving eHJB equations to obtain efficient state-dependent  Langevin dynamics for minimization problems.}
The work differs from the literature in that our goal is to obtain accurate state-dependent noise coefficients that depend on the Laplacian of the value functions—which solve the eHJB equation—rather than on the value functions themselves, e.g. in \cite{KimChoKimKim2025}.  
Three practical issues arise in the development of numerical methods:  
(i) 
A reliable evaluation of the Laplacian (trace of the Hessian), and  
(ii) stable computation of the log-partition term $-\log\!\int_{\mathcal U}\exp(-\frac{u}{\lambda}(\cdot))\,du$, approximating the minimum over $\mathcal{U}$ 
and 
(iii) the range of control for the noise magnitude or temperature.

First, the evaluation of the temperature requires the Hessian of the solution
to the HJB (eHJB) equations, which is continuous but not differentiable since 
the eHJB solution in \cite{GaoXuZhou2022} is locally $C^{2,\alpha}$, $\alpha\in (0,1)$.
Hence, additional regularization is essential. We adopt physics-informed neural networks (PINNs,\cite{RAISSI2019686}) to solve the eHJB equations, relying on their implicit regularization under stochastic gradient–based training. 
We show numerically that PINNs trained with the LION optimizer \cite{2023arXiv230206675C} yield solutions of eHJB equations that provide effective state-dependent temperatures for nonconvex optimization.
The above approach addresses the issue of (i). 

It is observed that 
 \emph{accurate} eHJB solutions are not required for effective optimization. Based on 
 Theorem \ref{thm:convergence-lambda-pertub}, the convergence of temperature is slow in the exploration parameter $\lambda$ and the training error. 
 Fortunately, we observe in Section~\ref{sec:numerical-examples} that the performance of 
 the Langevin dynamics with eHJB primarily depends on the \emph{relative shape} of the Laplacian rather than its precise values: nearly zero noise around global minimizers, and larger noise near local extrema and saddle points. PINNs provide smoothed approximations that capture the correct qualitative structure even when the underlying solution is not fully resolved.

Second, we develop an efficient stabilized implementation of the log-partition that avoids numerical integration while maintaining accuracy. %
The log-partition term  
has a closed-form expression for an interval $\mathcal U$, but direct evaluation can be unstable due to the exponential and logarithm functions. When their arguments are close to zero, we use Taylor's approximation to address (ii) (see Section \ref{ssec:stable-evaluation-z-h}).
Moreover, we use a truncation argument if the noise magnitude or temperature is very close to zero (see Section \ref{ssec:langevin-pinn-implementation}).

Third, we adopt a simple but effective scaling rule to determine the control set $\mathcal U=[u_{\min},u_{\max}]$, instead of manual selection \cite{GaoXuZhou2022}.
While the noise  coefficient  or temperature may remain positive, it may induce oscillations near global minimizers. To address this issue, we truncate the noise coefficients when they are below a 
dyadic fractions of $\sqrt{2u_{\max}}$ without harming the efficiency or the accuracy. 
See more details in Section \ref{ssec:choose-control-bound}.

In summary, we apply PINNs for exploratory HJB equations and use the Laplacian of the 
solutions to place the state-dependent temperature and noise coefficient. The novelty and key contributions of the work are listed below:
\begin{enumerate}
    \item \textbf{From theory and illustration in 1D to genuinely nonconvex minimization in higher dimensions.}
    We extend the computation using eHJB-based state-dependent temperature control from one dimension  \cite{GaoXuZhou2022} to $d=2$--$6$, where \emph{saddle points and local minimum} dominate the optimization landscape.
    The purpose of applying PINNs here is not to solve the eHJB equation pointwise, but to learn the structure of its Laplacian being small in the vicinity of global minimizers and large near saddle points and local minima. This perspective fundamentally differs from most existing HJB solvers, including the recent eHJB solver of \cite{KimChoKimKim2025}, which targets accurate pointwise solutions. In addition, our method solves the eHJB equation offline, requiring only a single run of PINNs, whereas \cite{KimChoKimKim2025} solves it on the fly and therefore invokes PINNs repeatedly.
    \item \textbf{A stable algorithm for elevating the nonlinear operator in eHJB.}
    We design an efficient, numerically stable implementation of the nonlinear log-partition operator in the eHJB equation, eliminating the need for numerical quadrature. See Section 
    \ref{ssec:stable-evaluation-z-h} for details. This algorithm allows stable computation in PINNs for eHJB directly and thus avoids the policy iterations and multiple runs of PINNs.

    \item \textbf{Control-range design and robust Langevin dynamics.}
    We provide in Section \ref{ssec:choose-control-bound} an empirical rule for selecting the control bounds that set the diffusion/noise coefficient (rather than hand-tuning as in \cite{GaoXuZhou2022}), and we stabilize the resulting dynamics with mirror reflection at the boundary of the computational domain and truncation of the noise coefficient to suppress oscillations near global minimizers while preserving exploration elsewhere.
\end{enumerate}

In this work, we consider the minimization problems on bounded domains, which admit no global minimizer on the boundary. 
The extension to unbounded domains is of its own interest and thus 
will be considered in a separate work.

The rest of the paper is organized as follows:
In Section \ref{sec:related-work}, we summarize works on Langevin dynamics, numerical methods for HJB equations and neural network methods for HJB equations and more generally nonlinear PDEs. 
We present in Section \ref{sec:setup} the problem formulation. 
In Section \ref{sec:numerics}, we present numerical methods for eHJB equations and the Langevin dynamics for minimization problems.
In Section \ref{sec:numerical-examples}, we present an example of solving eHJB equation and several examples of non-convex optimization. 
We summarize our work in Section \ref{sec:conclusion} and discuss some
limitations of the work and directions to be addressed.

\section{Literature review}
\label{sec:related-work}

In this section, we briefly review related works, including Langevin dynamics for optimization and classical grid-based methods and network-based methods for classical and exploratory HJB equations. At the end of the section, we discuss the key differences of the current work from the literature.


\textit{Langevin dynamics} provides a stochastic alternative to deterministic gradient descent for nonconvex optimization by injecting noise so that the induced dynamics favors global minimizers  rather than getting trapped in a single local basin, see e.g., 
 \cite{MR854068,MR942621}. 
 More recent theory has developed nonasymptotic guarantees for stochastic-gradient Langevin dynamics in high-dimensional nonconvex learning and optimization \cite{raginsky2017nonconvexlearningstochasticgradient,xu2020globalconvergencelangevindynamics}. In practice, an effective temperature (half of the square of the noise coefficient) is treated as a design parameter and found to improve efficiency  in deep learning \cite{neelakantan2015addinggradientnoiseimproves}. 
 Algorithmic temperature-control and multi-temperature strategies have been proposed, including explicit temperature control rules for annealing \cite{MunaNaka01} and replica-exchange mechanisms for coupling multiple temperatures\cite{dong2021replicaexchangenonconvexoptimization}. Also, the state-dependent temperature was designed in \cite{GaoXuZhou2022} to support effective exploration.

\textit{Classical numerical methods for HJB equations} are usually grid-based, which are expensive in high dimensions as the computational cost grows exponentially with the dimensionality of the physical domains, see e.g. \cite{CrandallIshiiLions1992,FlemingSoner2006,
FalconeFerretti2013,HuangForsythLabahn2012,FengLewisRapp2022}.
Also, these methods often explicitly use monotonicity, stability, and consistency, which impose additional barriers to the design of efficient numerical methods.

Neural-network PDE solvers have emerged as flexible, mesh-free alternatives to classical discretizations. Thus these methods can be extended to higher dimensions, e.g.,
physics-informed neural networks (PINNs) \cite{RAISSI2019686},
Deep Galerkin Method \cite{Sirignano_2018}, Backward stochastic differential equations (DeepBSDE)-based methods \cite{HanJentzenE2018,HurePhamWarin2020}, along with many subsequent variants.
In particular, there have been significant advances in solving high-dimensional PDEs, 
e.g.,  random alternating direction methods \cite{Hu_2024},
score-based diffusion \cite{hu2024scorebasedphysicsinformedneuralnetworks}, 
tensor neural networks \cite{wang2024tensorneuralnetworkshighdimensional}. In this work, we use PINNs for its simplicity. 

\textit{High-dimensional HJB equations} remain a central computational bottleneck in stochastic control literature. Existing neural approaches largely fall into two families: actor--critic schemes that learn both value and control representations, and physics-informed neural network (PINN) specialized to HJB equations, including policy-iteration PINNs and viscosity-solution-oriented formulations \cite{ZhouHanLu2021,NakamuraZimmererGongKang2021,MengEtAl2024PINNPI,LiuEtAl2023PINNHJB}. In actor--critic methods, the value function and control are parameterized by neural networks and trained via temporal-difference objectives and policy gradients, often with variance-reduction to mitigate Monte Carlo noise \cite{ZhouHanLu2021}. Despite substantial progress, reported high-dimensional demonstrations for classical HJB equations remain limited and frequently rely on weakly coupled dynamics or manufactured solutions for verification \cite{ZhouHanLu2021}.

The entropy-regularized/exploratory control framework provides an analytically and numerically attractive relaxation that replaces the pointwise extremum over controls by a log-partition(log-integral-exp) operator, yielding a distribution-valued (Gibbs) optimal policy \cite{WangZariphopoulouZhou2020}. The associated exploratory HJB equation has been studied for well-posedness and regularity, and its value function converges to that of the classical control problem as the exploration/regularization parameter $\lambda \to 0$ \cite{TangZhangZhou2022}. Beyond smoothing, the exploratory formulation can improve robustness by reducing sensitivity to noise and numerical perturbations, which is particularly relevant in annealing-type procedures \cite{GaoXuZhou2022}.
In \cite{KimChoKimKim2025}, 
PINN-based policy iteration has been applied to exploratory HJB equations, using an inner policy-iteration loop and approximating the value and policy separately.

In this work, we apply the framework of PINNs to solve the high-dimensional exploratory HJB equation \emph{directly}. 
The direct computation of eHJB avoids policy-iteration steps in  
\cite{dupret2026deep,KimChoKimKim2025} and thus saves the computational cost from multiple solves using PINNs. 
Once the eHJB equation is solved, we compute the noise coefficient based on the Hessian of the solution from PINNs, which leads to \textit{a state-dependent temperature} \cite{TangZhangZhou2022,GaoXuZhou2022} in Langevin dynamics. Thus, \textit{our computational goal in PINNs is totally different} from those in the literature, such as \cite{KimChoKimKim2025,LiuEtAl2023PINNHJB,dupret2026deep,MengEtAl2024PINNPI,ZhouHanLu2021}.


\section{Problem Setup}\label{sec:setup}

We consider the minimization problem
$
\min_{\bm{x}\in \Omega} f(\bm{x}),
$
where $\Omega\subset \mathbb{R}^d$ is a bounded  domain and
$f:\mathbb{R}^d\to\mathbb{R}$ is continuously differentiable.
For simplicity, we assume that the minima do not lie on the boundary $\partial \Omega$.

We denote by $\nabla := \nabla_{\bm{x}}$ the gradient with respect to $\bm{x}$ and by
$\nabla^2 := \nabla_{\bm{x}}^{2}$ the Hessian. We also write
$\Delta v := \mathrm{Tr}\!\big(\nabla^{2} v\big)$. 

Let us briefly recall the theory of   Langevin dynamics for   $\min_{\bm{x}\in \Real^d} f(\bm{x})$ and the determination of the noise coefficient via HJB. 
Adding  noise to the gradient descent method  
$X_{k+1}=X_k  - \eta_k \nabla f(X_k) $  leads to   Langevin dynamics, which 
help to escape local minima 
and saddle points. In the continuous time level, we have  
\begin{equation}\label{eq:opt-sde-rigid}
dX(t) = -\nabla f(X(t))\,dt + h (X(t))\,dW(t),\qquad X(0)=\bm{x} \in \Real^d,
\end{equation}
where $W_t$ is a $d$-dimensional Brownian motion and $\bm{x}$ is a guess of global minima.
The determination of the noise coefficient $h(\bm{x})$ can be formulated as controlled stochastic differential equations, which correspond
to   Hamilton-Bellman-Jacobi equations \cite{GaoXuZhou2022}. Specifically, consider the value function
\begin{equation}\label{eq:value-function}v(\bm{x})=\inf _{u \in \mathcal{A}_0(\bm{x})} \mathbb{E}\left[\int_0^{\infty} e^{-\rho t} f(X(t)) d t \mid X(0)=\bm{x}\right],
\end{equation}
where $\mathcal{A}_0$ be the set of admissible controls which may depend on the initial state   $\bm{x}$. The 
value function satisfies the classical HJB:
\begin{equation}\label{eq:classical-hjb}
-\rho v(\bm{x})+f(\bm{x})-\nabla f(\bm{x})\cdot\nabla v(\bm{x})+H(\nabla^2 v(\bm{x}))=0, \, H(\cdot)=\inf _{u\in \mathcal{U}}\left(u\operatorname{Tr}(\cdot)\right).
\end{equation}
Here $\mathcal{U}=[u_{\min},u_{\max}]$.
The corresponding optimal policy reads 
\begin{equation}\label{eq:class-opt-control}
 \bar{u}^*(\bm{x})=
u_{\min}  \text{ if }  \Delta v (\bm{x}) \geq 0;\text{ and is }    u_{\max}  \text{ otherwise,}
\end{equation} 
and the noise coefficient in the Langevin dynamics is  $h(X(t))=\sqrt{2 \bar{u}^*(X(t))}$.  

However, the noise coefficient is very sensitive to the sign of $\Delta v$ when $\Delta v$ is close to zero and thus is not computation friendly. A relaxed approach is proposed in \cite{GaoXuZhou2022} by approximating  the operator 
$H(\cdot)$ by 
\begin{equation}
 H_\lambda(\cdot) =-\lambda \ln \displaystyle\int_{\mathcal{U}}
\exp\!\Bigl(- {\lambda}^{-1}{\operatorname{Tr}(\cdot)}\, u\, \Bigr)\, \mathrm{d}u.
\end{equation} 
The resulting HJB equation, called 
  exploratory HJB equation in \cite{GaoXuZhou2022,TangZhangZhou2022}, reads
\begin{equation}\label{eq:exp-temp-lang}
F_\lambda(\nabla^2 v_\lambda, \nabla v_\lambda, v_\lambda, \bm{x})=:-\rho v_\lambda(\bm{x}) + f(\bm{x}) - \nabla v_\lambda(\bm{x})\cdot \nabla f(\bm{x})
+ H_\lambda(\nabla^2 v_\lambda(\bm{x})) =0, 
\end{equation}
The \emph{noise coefficient}  can be determined by 
\begin{eqnarray}\label{eq:noise-magnitude}
h_\lambda(\bm{x}): 
=\sqrt{( 2\int_{\mathcal{U}} u\,
\exp\!\left(-{\lambda}^{-1}\,u\,\Delta v_\lambda(\bm{x})\right)\,du)/{Z_\lambda(\bm{x})} }, 
Z_\lambda(\bm{x}) = \int_{\mathcal{U}}\exp\!\big(-{\lambda}^{-1}u\,\Delta v_\lambda(\bm{x})\big)\,du.
\end{eqnarray}
The induced state-dependent Langevin dynamics is
\begin{equation}\label{eq:opt-sde}
dX(t) = -\nabla f(X(t))\,dt + h_\lambda(X(t))\,dW(t),\qquad X(0)=\bm{x} \in\Real^d.
\end{equation}
%
It is shown in \cite{TangZhangZhou2022} that when $\lambda\to 0$, $v_\lambda$ converges to the value function $v$ in $C^{2,\alpha}$, for some $\alpha\in (0,1)$. 

For the minimization problem $\min_{\bm{x}\in\Omega}f(\bm{x})$, 
we still apply the above theory 
while we set $\bm{x}\in \Omega$.   For the corresponding HJB equations,  we apply
 the homogeneous Neumann boundary condition
$
\nabla v_\lambda(\bm{x})\cdot n(\bm{x})=0, \bm{x}\in \partial\Omega.
$
Correspondingly, the Langevin dynamics should be imposed with reflection boundary conditions, which we discuss in the next section.

\section{Methodology}\label{sec:numerics}
In this section, we present numerical methods solving eHJB through PINNs and the corresponding Langevin dynamics. 
We summarize in Table~\ref{tab:notation_numerics}  the key notation used in  Section~\ref{sec:numerics}.

\begin{table}[t]
\centering
\caption{Key notation in Section~\ref{sec:numerics}.}
\label{tab:notation_numerics}
\renewcommand{\arraystretch}{1.1}
\setlength{\tabcolsep}{4pt}
\begin{tabular}{ll}
\hline
\textbf{Symbol} & \textbf{Meaning} \\
\hline
$\mathcal R_\lambda(\bm{x};\phi)$ & PDE residual of the eHJB equation at $\bm{x}$ for network parameters $\phi$ \\
$N_\Omega$ & number of interior collocation points in $\Omega$   \\
$N_{\partial\Omega}$ & number of boundary collocation points on $\partial\Omega$  \\
$\alpha_{\mathrm{res}}$ & weight of the interior PDE residual term in \eqref{eq:loss_infinite} \\
$\alpha_{\partial\Omega}$ & weight of the boundary-condition term in \eqref{eq:loss_infinite} \\

$v_{\lambda,\phi}$ & PINN approximation of the exploratory value function $v_\lambda$ \\
$h_{\lambda,\phi}$ & PINN-induced noise coefficient computed from $\Delta v_{\lambda,\phi}$ \ \eqref{eq:noise-magnitude} \\
$\eta$ & step size in the Euler--Maruyama/Langevin iteration \\
$\tau$ & truncation threshold for avoiding numerical instability \\
$h_{\lambda,\phi}^\tau$ & truncated noise coefficient: $h_{\lambda,\phi}^\tau(\bm{x})=h_{\lambda,\phi}(\bm{x})\mathbf{1}_{\{h_{\lambda,\phi}(\bm{x})\ge\tau\}}$ \\
\hline
\end{tabular}
\end{table}

\subsection{PINNs for the exploratory HJB equation}\label{subsec:model-residuals}

We solve \eqref{eq:exp-temp-lang} with physics-informed neural networks (PINNs)~\cite{RAISSI2019686}.
Let $v_{\lambda,\phi}:\Omega\to\mathbb{R}$ be a neural network. 
Denote 
\begin{equation}\label{eq:residualboundary-elliptic}
\mathcal R_\lambda(\bm{x};\phi):=F_\lambda(\nabla^2 v_{\lambda,\phi},\nabla v_{\lambda,\phi},v_{\lambda,\phi},\bm{x}),
\qquad
\mathcal B_{\partial\Omega}(\bm{x};\phi):=\nabla v_{\lambda,\phi}(\bm{x})\cdot n(\bm{x}).
\end{equation}
%
Let $D_{\Omega}=\{\bm{x}_i\}_{i=1}^{N_\Omega}\subset\Omega$ and
$D_{\partial\Omega}=\{\bm{x}_\ell\}_{\ell=1}^{N_{\partial\Omega}}\subset\partial\Omega$
be  uniformly sampled collocation points. We find the parameters of a feedforward neural networks by minimizing  the following loss 
\begin{equation}\label{eq:loss_infinite}
\mathcal L(\phi;D_{\Omega},D_{\partial\Omega})
:=\alpha_{\mathrm{res}}\frac{1}{N_\Omega}\sum_{i=1}^{N_\Omega}\big|\mathcal R_\lambda(\bm{x}_i;\phi)\big|^2
+\alpha_{\partial\Omega}\frac{1}{N_{\partial \Omega}}\sum_{\ell=1}^{N_{\partial \Omega}}\big|\mathcal B_{\partial\Omega}(\bm{x}_\ell;\phi)\big|^2.
\end{equation}
Here, $\alpha_{\mathrm{res}},\alpha_{\partial\Omega}>0$ are user-chosen weights that balance the two terms in the loss. We briefly recall the training/optimization of the loss, in Algorithm~\ref{alg:exploratory-pinn}.

To estimate \textit{the errors of PINNs}, we use the following 
\begin{equation}\label{eq:exploratory-hjb-elliptic-pertubed}
F_\lambda(\nabla^2 v_{\lambda,\phi}, \nabla v_{\lambda,\phi}, v_{\lambda,\phi}, \bm{x})\;=\mathcal R_\lambda(\bm{x};\phi):=\varepsilon \widetilde{R}_{\lambda}(\bm{x};\phi).
\end{equation}
Here we set $\norm{\widetilde{R}_\lambda}_{L^\infty(\Omega)}=1$  and   $\varepsilon>0$.
Since we use $\Delta v_{\lambda,\phi}$ and \eqref{eq:noise-magnitude} instead of using the control \eqref{eq:class-opt-control}, two types of errors have been induced by $\lambda$ and $\varepsilon$ in the Langevin dynamics. 
\begin{thm}\label{thm:convergence-lambda-pertub}
Let $\Omega \subset\Real^d$ be an open bounded domain with a smooth boundary.
Assume that $f\in C^2(\Omega)$ and $\widetilde{R}_\lambda(\bm{x};\phi)\in L^\infty (\Omega)$ and is uniformly bounded in $\lambda$. Let $v$  and 
$v_{\lambda,\phi}$ satisfy solutions to 
\eqref{eq:classical-hjb} and \eqref{eq:exploratory-hjb-elliptic-pertubed}.
There exists $C>0$ independent of $\lambda$ and $\varepsilon$ such that  
\begin{equation}
u_{\min}\norm{\Delta v_{\lambda,\phi}-\Delta v}_{L^\infty(\Omega)}\leq C(\lambda  +\varepsilon \norm{\widetilde{R}_\lambda(\cdot,\phi)}_{L^\infty} ).
  \label{eq:convergence-derivatives-infty-final}
\end{equation} 
\end{thm} 
The residual $\widetilde{R}_\lambda \in L^\infty(\Omega)$ is uniformly bounded when we use smooth activation (thrice-continuously differentiable) functions in the network $v_{\lambda,\phi}$. 
\begin{rem}\label{rem:order-convergence}
Here we consider the a posteriori error of the PINNs in this theorem. After training, $\widetilde{R}_\lambda$ and $\varepsilon$ can be evaluated. 
If the error from the boundary loss is negligible, the error of PINNs in second-order derivatives can be estimated as in Theorem 
\ref{thm:convergence-lambda-pertub}, based on \eqref{eq:loss_infinite}.  For the approximation (a priori) error of PINNs for nonlinear PDEs, we refer to \cite{arakelyan2024convergence}.

The estimate in Theorem 
\ref{thm:convergence-lambda-pertub} does not give a convergence order in $\lambda$. 
According to the approximation theory and the fact that $\Delta v \in C^{0,\alpha}$ with $\alpha\in (0,1)$, the best convergence order one can expect is then $\alpha$. 
In Example \ref{exm:1d-pinn-hjb}, we numerically show that the convergence order in $\lambda$ is half. 
\end{rem}

\subsection{Stable evaluation of $H_\lambda(\nabla^2 v_{\lambda,\phi})$ and $h_{\lambda,\phi}$} \label{app:h-zimp}
The log-partition term in $\mathcal R_\lambda$ and the induced noise coefficient $h_{\lambda,\phi}$ both involve
$\int_{\mathcal U}\exp(-u\,\Delta v_{\lambda,\phi}/\lambda)\,du$, which can overflow/underflow if evaluated naively.
Below, we present numerically stable evaluations. 

\begin{algorithm}[t]
  \caption{Solving eHJB Equation with PINN}
  \label{alg:exploratory-pinn}
  \begin{algorithmic}[1]
    \Require Domain $\Omega$; parameters $(\rho,\lambda,u_{\min},u_{\max})$; weights $(\alpha_{\mathrm{res}},\alpha_{\partial\Omega})$;
    iterations $T$; batch sizes $(N_\Omega,N_{\partial\Omega})$; learning rate $\{\psi_t\}$.
    \State Initialize network parameters $\phi_0$ (and optimizer state).
    \For{$t=0,1,\dots,T-1$}
      \State Sample $D_\Omega^{(t)}\subset\Omega$ and $D_{\partial\Omega}^{(t)}\subset\partial\Omega$.
      \State Compute $\mathcal L(\phi_t;D_\Omega^{(t)},D_{\partial\Omega}^{(t)})$ using \eqref{eq:residualboundary-elliptic}--\eqref{eq:loss_infinite} 
      \State Update $\phi_{t+1}\leftarrow \phi_t - \psi_t\,\nabla_\phi \mathcal L(\phi_t;D_\Omega^{(t)},D_{\partial\Omega}^{(t)})$
      (e.g., Adam/SGD/LION).
    \EndFor
    \State \Return $\phi\leftarrow \phi_T$.
  \end{algorithmic}
\end{algorithm}

\label{ssec:stable-evaluation-z-h}
Recall $\mathcal{U} = [u_{\min}, u_{\max}]$, with $u_{\max} > u_{\min} > 0$.
Define $\delta = u_{\max} - u_{\min}$ and
$z = -{\delta}{\lambda^{-1}}\Delta v_{\lambda, \phi}(\bm{x})$.
Then we have, by direct calculations,  
\begin{align}
  H_\lambda(\nabla^2 v_{\lambda,\phi}) & =  \begin{cases}-\lambda\left( \frac{u_{\min}}{\delta} z + \ln(\delta) +\ln \big( \frac{\exp(z) - 1}{z}\big) \right) & z \neq 0;\\ 
    -\lambda\ln(\delta) & z = 0.
\end{cases}\\
h_{\lambda, \phi}(\bm{x}) &= \begin{cases}
    \sqrt{
      2 \left(u_{\min} + \delta\frac{(z - 1)\exp(z) + 1}{z\left(\exp(z) - 1 \right)}
    \right) } & z \neq 0,\\
    \sqrt{u_{\max} + u_{\min}} & z = 0.
\end{cases}
\label{eq:noise-trans}
\end{align}
Let $\epsilon_0>0$ be small. For $|z|<\epsilon_0$, we evaluate $\ln ( \frac{\exp(z) - 1}{z})$ via its Taylor approximation to avoid cancellation in $\exp(z)-1$.
To implement $h_{\lambda, \phi}$, 
we approximate $\frac{(z - 1)e^z + 1}{z(e^z - 1)}$   by 
 $
\displaystyle \tfrac12 + \tfrac{z}{12} - \tfrac{z^{3}}{720} + \tfrac{z^{5}}{30240}, \text{ if } |z|<\epsilon_0$, 
and $\dfrac{(z-1)+e^{-z}}{\,z\,(1-e^{-z})\,}$, when $z>\epsilon_0$, and 
$
\dfrac{z + (z-1)\big(e^{z}-1\big)}{\,z\big(e^{z}-1\big)\,}$, when  $z<-\epsilon_0$.

\subsection{Choice of the control set $\mathcal{U}$ and the discount rate $\rho$}\label{ssec:choose-control-bound}

In Equations \eqref{eq:classical-hjb} and \eqref{eq:exp-temp-lang}, we need to choose the control set $\mathcal{U}=[u_{\min},u_{\max}]$.  As we want the noise coefficient (and the temperature $\bar{u}^*$) to be small around the global minima and large around the saddle points and local minima, we may set 
$u_{\min}$ to be zero and $u_{\max}=\infty$. However, this configuration brings some difficulties in theory and simulations. First of all, $u_{\min}=0$ leads to degenerate elliptic equations, which require further theoretical development beyond \cite{TangZhangZhou2022} in the convergence of the solution in $\lambda$. 
Second, it is difficult to set $u_{\max}=\infty$ or even a very large number in modern computers that use double precision, since the ratio $u_{\max}/u_{\min}$ must remain relatively small for numerical stability.

Our choice of the control set is based on computational efficiency and stability. In the Langevin dynamics for minimization problems, the noise coefficient $h_\lambda$ should be proportional to the gradient $\nabla f$.  Also, $h_\lambda$ should be close to zero around global minima and not too large for computational efficiency. 
Moreover, $u_{\min}$ should not be too close to zero to avoid instability.
If $\nabla f$ is close to zero, a very small noise in the discrete Langevin dynamics will cause oscillations in the dynamics and thus deteriorate the convergence of the Langevin dynamics. 
Combining all these considerations, we set 
$u_{\min}>0$   and pick $u_{\max}$ to be proportional to the magnitude of $\nabla f$.
In practice, we use an estimated magnitude of 
$\nabla f$. Specifically,
 we draw $S$ samples $x^{(l)}\sim \mathrm{Unif}(\Omega)$ and define
\begin{equation}\label{eq:kappa-beta}
\kappa := \frac{1}{2}\Big(\max_{1\le l\le S}\|\nabla f(x^{(l)})\|_\infty\Big)^2,
\qquad
u_{\max} := C_\kappa\,\kappa,\;\; C_\kappa>0.
\end{equation}
By \eqref{eq:noise-magnitude},   $h_\lambda\leq \sqrt{2u_{\max}}=\sqrt{C_\kappa \kappa}$. 
While we don't have a rule to set the lower bound $u_{\min}$, we find that
$u_{\min}\leq 10^{-2}$ and $u_{\min}\geq 10^{-8}$ ($h_\lambda\geq \sqrt{2u_{\min}}\geq 10^{-4}$) suffices in all our examples in Section \ref{sec:numerical-examples}.  

From the definition 
\eqref{eq:value-function}, we can readily observe that 
the discount rate $\rho$ is 
related to the scaling of `time' $t$. With a change of variable $\rho t =s$, the scaling $\rho$ will appear as coefficients in the drift and diffusion coefficient in the Langenvin dynamics \eqref{eq:opt-sde-rigid} and \eqref{eq:opt-sde} and later in the discrete Lagenvin dynamics 
\eqref{eq:truncated}. 
In this work, we will not study the effect of the discount rate $\rho$ theoretically. Instead, we keep $\rho $ at the order of one and test the effect of $\rho$ in numerical examples of Section \ref{sec:numerical-examples}.

\subsection{Langevin dynamics with approximated temperature}
\label{ssec:langevin-pinn-implementation}
Given   $v_{\lambda,\phi}$, we evaluate the induced
noise coefficient $h_{\lambda,\phi}$ in  \eqref{eq:noise-magnitude} and define the
state-dependent Langevin dynamics
\begin{equation}
    d X^*(t)=-\nabla f\left(X^*(t)\right)\, d t+h_{\lambda, \phi}\left(X^*(t)\right)\, d W_t,
    \quad X^*(0)= \bm{x}.
    \label{eq:opt-sde-pinn}
\end{equation}
%
Near the minimizers, $h_{\lambda,\phi}$ may remain small but nonzero,
which can lead to persistent oscillations. To reduce such oscillations, we introduce the truncated coefficient
\begin{equation}\label{eq:truncated-h}
  h_{\lambda, \phi}^\tau(\bm{x}) \;=\;
    h_{\lambda, \phi}(\bm{x})\,\mathbf{1}_{\{h_{\lambda,\phi}(\bm{x})\ge \tau>0\}}.
\end{equation}
The truncation \eqref{eq:truncated-h}--\eqref{eq:truncated} is a numerical stabilization heuristic that suppresses weak noise below $\tau$
and improves empirical convergence (see Example \ref{exm:double-well}).
We use the Euler--Maruyama scheme to obtain the discrete Langevin dynamics
\begin{equation}
\label{eq:truncated}
X_{k+1}^* \;=\; X_k^* \;-\; \eta\,\nabla f(X_k^*)
\;+\; \sqrt{\eta}\,h_{\lambda, \phi}^\tau(X_k^*)\,\xi_k,\, X^*(0)=\bm{x},
\quad \xi_k\sim\mathcal N(0,I_d).
\end{equation}

 %
To enforce $X^*_k\in\Omega$, we use mirror reflection. In this work,  we use 
$\Omega=\prod_{i=1}^d[a_i,b_i]$ and apply a coordinate-wise mirror map after each iteration, see 
Algorithm~\ref{alg:mirror-boundary}.

\begin{algorithm}[H]
\caption{State-dependent Langevin dynamics on $\Omega$ (with mirror reflection)}
\label{alg:mirror-boundary}
\begin{algorithmic}[1]
\State \textbf{Input:} step size $\eta>0$; horizon $T$; $X^*_0\in\Omega$; $f$, $\nabla f$; $h_{\lambda, \phi}$ (or $h^\tau_{\lambda, \phi}$).
\For{$k=0,1,\dots,T-1$}
  \State Sample $\xi_k \sim \mathcal{N}(0,I_d)$.
  \State Compute $X_{k+1}^*$ by \eqref{eq:truncated}.
  \State \textbf{Mirror reflection:} $X^*_{k+1} \leftarrow \mathrm{Mirr}_{\Omega}(X_{k+1}^*)$ with
  \[
    \big(\mathrm{Mirr}_{\Omega}(\bm{x})\big)_i =
    \begin{cases}
      a_i + w_i r_i, & \lfloor z_i \rfloor \text{ even},\\
      b_i - w_i r_i, & \lfloor z_i \rfloor \text{ odd},
    \end{cases}\;
\substack{
    w_i=b_i-a_i,\;\\
    z_i=\frac{x_i-a_i}{w_i},\;}
    r_i=z_i-\lfloor z_i \rfloor\in[0,1).
  \]
  \State Record $X^*_{k+1}$ and $f(X^*_{k+1})$.
\EndFor
\State \textbf{Output:} $\{X^*_k\}_{k=0}^T$ and $\{f(X^*_k)\}_{k=0}^T$.
\end{algorithmic}
\end{algorithm}



We now summarize our approach in Algorithm~\ref{alg:nonconvex-optimization} 
for solving minimization problems.
For the objective $f$, we first solve the eHJB equation with  PINNs and loss  \eqref{eq:loss_infinite} to obtain $\Delta v_{\lambda,\phi}$ and then compute the induced temperature  $h_{\lambda,\phi}^\tau$ \eqref{eq:truncated-h}; finally, we run a state-dependent Langevin dynamics
\eqref{eq:truncated}
on $\Omega$ with mirror reflection on the boundary to generate candidate global minimizers.

\begin{algorithm}[t]
\caption{Solving Non-Convex Optimization via Exploratory HJB Equations and PINNs}
\label{alg:nonconvex-optimization}
\begin{algorithmic}[1]
\Require $f$ and $\nabla f$ on $\Omega$; $\mathcal U=[u_{\min},u_{\max}]$; $(\rho,\lambda)$; PINN settings; Langevin settings $(\eta,T,\tau)$.
\Ensure Approximate minimizer of $f$ and state-dependent temperature $h_{\lambda,\phi}$.
\State Solve the eHJB equation on $\Omega$ with PINN (Algorithm~\ref{alg:exploratory-pinn}) to obtain $v_{\lambda,\phi}$ and $\Delta v_{\lambda,\phi}$.
\State Compute $Z_{\lambda,\phi}$ and the noise coefficient $h_{\lambda,\phi}^\tau$ from $\Delta v_{\lambda,\phi}$.
\State Run the state-dependent Langevin iterations on $\Omega$ with mirror reflection (Algorithm~\ref{alg:mirror-boundary}) to generate $\{X_k^*\}_{k=0}^T$.
\State \Return best candidate point in the trajectory, e.g.\ $\arg\min_k f(X_k^*)$ or reaching the maximum iteration.
\end{algorithmic}
\end{algorithm}


\section{Numerical Examples}\label{sec:numerical-examples}

In this section, we present a sequence of tests: in Examples \ref{exm:1d-pinn-hjb} and \ref{exm:double-well},  we first validate the PINN solver on a stationary eHJB equation  and the computed noise coefficient; 
then we apply Algorithm \ref{alg:mirror-boundary}  to four benchmark minimization problems. 
We summarize the key notation of this section in Table \ref{tab:notation_numerical_examples}.

\begin{table}[t]
\centering
\caption{Key notation in Section~\ref{sec:numerical-examples}.}
\label{tab:notation_numerical_examples}
\footnotesize
\setlength{\tabcolsep}{4pt}
\renewcommand{\arraystretch}{1.05}
\begin{tabular}{ll}
\hline
\textbf{Symbol} & \textbf{Meaning} \\
\hline
$N_{\mathrm{traj}}$ & number of simulated trajectories in Algorithm~\ref{alg:mirror-boundary} \\
$k$ & Langevin iteration index in Algorithm~\ref{alg:mirror-boundary}, $0\leq k\leq T-1$ \\
$\hat f_k$ & trajectory-averaged objective at iteration $k$, \ \eqref{eq:evol-eval} \\
$\mathcal{E}(\bm{x})$ & distance to the (global or specified target) minimizer set \\
$\mathcal{E}_k$ &  $\mathcal{E}(\bm{x})$ at iteration $k$, \ \eqref{eq:evol-eval} \\
$\kappa$ & gradient-based reference scale used to set $u_{\max}$, \ \eqref{eq:kappa-beta} \\
$C_\kappa$ & scaling factor in $u_{\max}=C_\kappa\,\kappa$ \\
$s$ & truncation fraction in $\tau=s\sqrt{2u_{\max}}$ \\
\hline
\end{tabular}
\end{table}

For all the examples of minimization, we run Algorithm~\ref{alg:mirror-boundary} with $N_{\mathrm{traj}}=100$\footnote{We also experiment with $N_{traj} = 1000$ and obtain the similar results.} trajectories and uniformly sample the initial state in 
\eqref{eq:truncated} on $\Omega$.
We take $T=1000$ in the discrete Langevin dynamics \eqref{eq:truncated}. 
Define, at 
the $k$-th iteration in \eqref{eq:truncated},
\begin{equation}\label{eq:evol-eval}
\hat f_k := \frac{1}{N_{\mathrm{traj}}}\sum_{j=1}^{N_{\mathrm{traj}}} f(X_{j,k}^*),
\qquad
\mathcal{E}_k := \frac{1}{N_{\mathrm{traj}}}\sum_{j=1}^{N_{\mathrm{traj}}} \mathcal{E}(X_{j,k}^*),  1\leq k\leq T,
\end{equation}
where $\mathcal{E}(\bm{x}):=\min_{1\le i\le q}\|\bm{x}-\bm{x}_{\star,i}\|_2$, and $\{\bm{x}_{\star,i}\}_{i=1}^q$ are the global minima (or a specified target minimizer for a given instance). To estimate $u_{\max}$, we set $S=2000$ in \eqref{eq:kappa-beta}. 
In \eqref{eq:truncated-h}, we set $\tau=s\sqrt{2u_{\max}}$ with $s\in(0,1]$ for the truncation of the noise coefficient.
%

In all runs of PINNs, we use float64 and a fully connected network with 5 hidden layers, width 64, and the $\tanh$ activation function. We use automatic differentiation to obtain gradients and Hessians. All networks are initialized with i.i.d. uniform random weights scaled by the square root of the layer input size, and with zero biases.
We train for $20000$ iterations using loss \eqref{eq:loss_infinite} with  $N_\Omega=16384$ and the LION optimizer~\cite{2023arXiv230206675C}.
Boundary collocation points depend on the domain geometry and are specified in each example.
At each iteration,  collocation points are resampled uniformly in $\Omega$ and on $\partial\Omega$.
We fix $\alpha_{\mathrm{res}}=1/\rho$ and $\alpha_{\partial\Omega}=50$ in \eqref{eq:loss_infinite}.

\begin{exm}[Testing PINNs for eHJB equation]
\label{exm:1d-pinn-hjb}
Consider the stationary equation
\begin{equation}\label{eq:exp-hjb-solve}
-\rho\, v_\lambda(\bm{x}) + g(\bm{x}) + H_\lambda (\nabla^2 v_{\lambda}(\bm{x}))
 = 0,\, \bm{x} \in \Omega; \,
\nabla_{\bm{x}} v_\lambda(\bm{x}) \cdot \bm{n}(\bm{x}) = 0,\, \bm{x} \in \partial \Omega,
\end{equation}
where $\Omega=[-3,3]^d$, $\bm{x}=(x_1,\dots,x_d)^\top$, and $\rho=1$, and $\mathcal U=[u_{\min},u_{\max}]$ with  $u_{\min}=0.2$, and $u_{\max}=1$. 
\end{exm}
The corresponding classical HJB equation reads 
\begin{equation}\label{eq:hjb-solve}
-\rho\, v(\bm{x}) + g(\bm{x}) + \min_{u\in\mathcal U}\bigl(u\,\Delta v(\bm{x})\bigr)=0,\qquad \bm{x}\in\Omega,
\quad
\nabla v(\bm{x})\cdot \bm n(\bm{x})=0,\quad \bm{x}\in\partial\Omega.
\end{equation}
We pick $
v(\bm{x})=\prod_{i=1}^d \cos({\pi x_i}/{3})$,
and  $g(\bm{x})$ can be found according to Equation \eqref{eq:hjb-solve}.

In this example, we test whether PINNs can obtain the Laplacian of the solutions to this HJB  and corresponding eHJB equations \footnote{In this example, the equation is different from \eqref{eq:exp-temp-lang} but the error estimate in Theorem \ref{thm:convergence-lambda-pertub} carries over.} accurately.  
Let $v_{\lambda,\phi}$ be the PINN solution of \eqref{eq:exp-hjb-solve}.
We measure the accuracy in $\Delta v_{\lambda,\phi}$ on uniformly sampled test points $\{\bm{x}_i\}_{i=1}^{N_{\mathrm{test}}}$ within $\Omega$ using the relative discrete $l_2$-error and the  $l_\infty$-error
\[
e_{l^2,r}
=
\frac{\Big(\sum_{i=1}^{N_{\mathrm{test}}}|\Delta v(\bm{x}_i)-\Delta v_{\lambda,\phi}(\bm{x}_i)|^2\Big)^{1/2}}
{\Big(\sum_{i=1}^{N_{\mathrm{test}}}|\Delta v(\bm{x}_i)|^2\Big)^{1/2}},
\quad
e_{l^\infty}=\max_{1\le i\le N_{\mathrm{test}}}|\Delta v(\bm{x}_i)-\Delta v_{\lambda,\phi}(\bm{x}_i)|.
\]

We set the number of the boundary collocation points $N_{\partial\Omega}=1024$ in  \eqref{eq:loss_infinite};
and $\lambda=0.32\times {2^{-p}}, p = 0, 1,\ldots,4$.  

For $d=1$, we present in Table~\ref{tab:laplacian-l_infty} the discrete maximum error $e_{l^\infty}$ and the $l_\infty$ residual on test points $\{\bm{x}_i\}_{i=1}^{N_{\mathrm{test}}}$ and  
$\varepsilon(\lambda) := \max_{1 \leq i \leq N_{\mathrm{test}}} \big|\mathcal R_\lambda(\bm{x}_i;\phi)\big|$.
%
We also present the ratio $e_{l^\infty}(2 \lambda)/e_{l^\infty}(\lambda)$, which is close to $\sqrt{2}$. This suggests a half-order convergence in $\lambda$, which aligns with the expectation outlined in Remark \ref{rem:order-convergence}.  In contrast, the ratio ${\varepsilon}(2 \lambda)/{\varepsilon}(\lambda)$  does not exhibit a monotone decrease.  However, the error ${\varepsilon}$ from PINNs is smaller than the error induced by the regularization parameter $\lambda$. 
 This suggests that  $\lambda$ is the primary source of error and  \textit{training PINNs to high accuracy is unnecessary}.

We present the relative $l_2$ errors in Table~\ref{tab:laplacian-relL2},  when $d=1,2,4$. 
The relative errors herein suggest that 
$\Delta v_{\lambda,\phi}$ preserves the overall spatial structure of $\Delta v$ required in the state-dependent noise coefficient 
$h_\lambda$.

\begin{figure*}[t]
\centering

\begin{minipage}[t]{0.48\textwidth}\vspace{0pt}
\centering
\includegraphics[width=\linewidth]{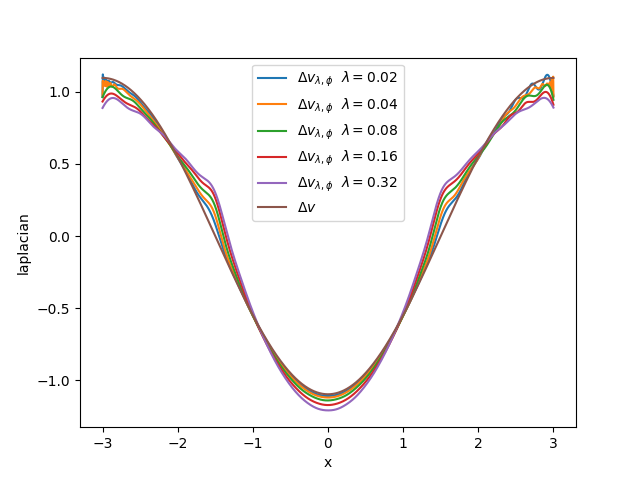}
\captionsetup{font=small} 
\caption{
    One-dimensional test: $\Delta v$ and $\Delta v_{\lambda,\phi}$ for
    $\lambda =\frac{0.32}{2^{p}},\,p=0,1,\ldots, 4$.
}
\label{fig:laplacian-d1}
\end{minipage}%
\hfill%
\begin{minipage}[t]{0.50\textwidth}\vspace{0pt}
\centering

\captionsetup{type=table,font=footnotesize} 

{\footnotesize
\setlength{\tabcolsep}{2pt}   
\renewcommand{\arraystretch}{1.0} 

\captionof{table}{Exploration level $\lambda$, discrete $l_\infty$ error $e_{l^\infty}$ between $\Delta v$ and $\Delta v_{\lambda,\phi}$ and corresponding residual ${\varepsilon}$ for $d = 1$ and $N_{\mathrm{test}} = 4096$.}
\label{tab:laplacian-l_infty}
\vspace{0.5ex}
\begin{tabular}{cccccc}
\hline
$\lambda$ & 0.32 & 0.16 & 0.08 & 0.04 & 0.02 \\
$e_{l^\infty}$ & 0.304 & 0.265 & 0.208 & 0.149 & 0.124 \\
$\frac{e_{l^\infty}(2 \lambda)}{e_{l^\infty}(\lambda)}$ & -- & 1.145 & 1.277 & 1.398 & 1.195 \\
${\varepsilon}$ & 0.026 & 0.025 & 0.022 & 0.020 & 0.023 \\
$\frac{{\varepsilon}(2 \lambda)}{{\varepsilon}(\lambda)}$ & -- & 1.027 & 1.132 & 1.088 & 0.893 \\
\hline
\end{tabular}

\vspace{1.2ex}

\captionof{table}{Relative  error $e_{l^2,r}$,  $N_{\mathrm{test}}=4096$.}

\label{tab:laplacian-relL2}
\vspace{0.5ex}
\begin{tabular}{cccccc}
    \hline
    $d$ & $\lambda = 0.32$ & $\lambda = 0.16$ & $\lambda = 0.08$ & $\lambda = 0.04$ & $\lambda = 0.02$ \\
    \hline
    1 & 0.167 & 0.134 & 0.094 & 0.059 & 0.036 \\
    2 & 0.203 & 0.152 & 0.101 & 0.062 & 0.036 \\
    4 & 0.255 & 0.196 & 0.141 & 0.110 & 0.102 \\
    \hline
\end{tabular}
} 

\end{minipage}

\end{figure*}

\begin{exm}[1D double well, \cite{GaoXuZhou2022}, evaluating noise coefficient]\label{exm:double-well}
Consider the nonconvex objective on $\Omega=[-6,6]$, $
f(\bm{x})=
(4x-20)\mathbf{1}_{x>6}
+(x-4)^2\mathbf{1}_{2<x\le 6}
+(8-x^2)\mathbf{1}_{-2<x\le 2}
+\bigl(2(x+3)^2+2\bigr)\mathbf{1}_{-6<x\le -2}
-(12x+52)\mathbf{1}_{x\le -6}
$, which admits a global minimizer at $x=4$ and a secondary local minimum at $x=-3$.
\end{exm}

We obtain $\kappa=71$ from \eqref{eq:kappa-beta} and fix $\rho=0.4$, $u_{\min}=0.2$, and $u_{\max}=2\kappa=142$.
As in the last example, we evaluate $v_{\lambda,\phi}''$ and $h_{\lambda,\phi}$ for
$\lambda =0.01\times  2^{p}$, $p=0,1,\ldots\,5$.

We compute a reference solution of the classical HJB equation~\eqref{eq:classical-hjb}
on $\Omega$ using a monotone finite-difference discretization (central differences for $v''$
and upwinding for the drift term), solved by Howard policy iteration, e.g., in \cite{MR2551155}.  
The reference  $v''$ is obtained by applying the same
central-difference operator to the converged grid solution, and the corresponding classical
noise coefficient is $\sqrt{2\bar u^*}$ from~\eqref{eq:class-opt-control}.

From
Figure~\ref{fig:ex1-three}(b), we obtain that    $v_{\lambda,\phi}''$ (from PINNs) closely matches the
reference   $v''$ for various $\lambda$'s. Moreover, $h_{\lambda,\phi}$ reproduces the qualitative shape of
$\sqrt{2\bar u^*}$ and attains its minimum near the global minimizer $x=4$; see Figure~\ref{fig:ex1-three}(c). 

\begin{figure}[t]
  \centering
  \begin{minipage}[t]{0.32\textwidth}
    \centering
    \includegraphics[width=\textwidth]{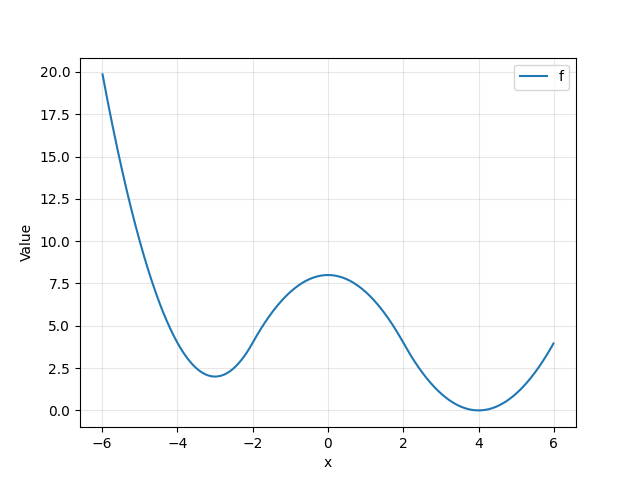}
    {\scriptsize (a) Objective $f$}
  \end{minipage}\hfill
  \begin{minipage}[t]{0.32\textwidth}
    \centering
    \includegraphics[width=\textwidth]{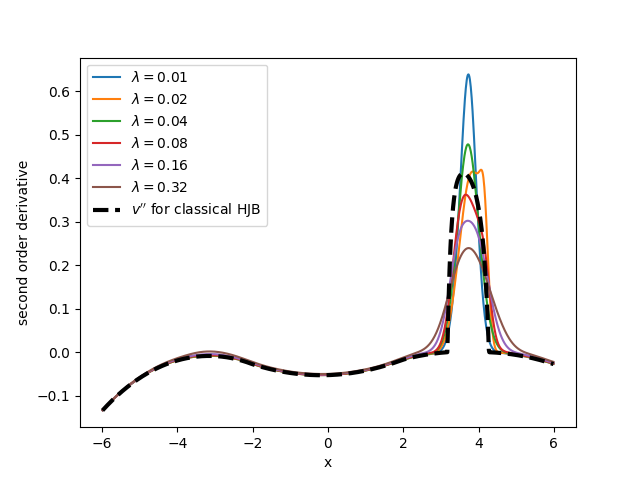}
    {\scriptsize (b) $v_{\lambda,\phi}''$ vs.\ $v''$}
  \end{minipage}\hfill
  \begin{minipage}[t]{0.32\textwidth}
    \centering
    \includegraphics[width=\textwidth]{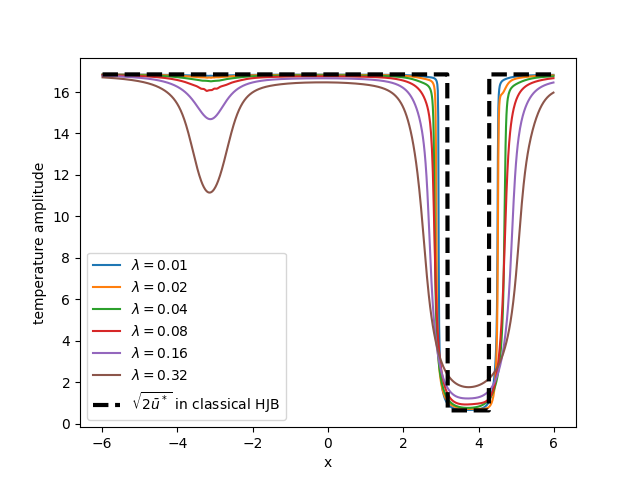}
    {\scriptsize (c) $h_{\lambda,\phi}$ vs.\ $\sqrt{2\bar{u}^*}$}
  \end{minipage}

  \caption{\scriptsize Example~\ref{exm:double-well} (1D double-well). (a) Objective $f$.
  (b) $v_{\lambda,\phi}''$ for multiple $\lambda$ compared with the reference $v''$ (Appendix~\ref{app:fd-howard}).
  (c) $h_{\lambda,\phi}$ compared with $\sqrt{2\bar{u}^*}$ from~\eqref{eq:class-opt-control}.}
  \label{fig:ex1-three}
\end{figure}

In Algorithm \ref{alg:mirror-boundary}, we take the step size $\eta=0.016$.
In this example, $h_{\lambda,\phi}$ remains positive near $x=4$. Thus, without truncation ($\tau=0$ in \eqref{eq:truncated-h}), 
the dynamics oscillates  and the trajectory-averaged objective $\hat f_k$ exhibits sustained large fluctuations
(Figure~\ref{fig:1d-trunc-opt}, left). With truncation $\tau=\tfrac12\sqrt{2u_{\max}}$, the trajectory are 
stable and converge for all tested $\lambda$'s within $200$ iterations (Figure~\ref{fig:1d-trunc-opt}, right).
Thus, accurate recovery of the \emph{shape} of the curvature/noise coefficient is sufficient for the considered minimization problem and the  truncation
provides a simple mechanism for stable minimization.

\begin{figure}[t]
  \centering
  \resizebox{0.9\textwidth}{!}{%
    \begin{minipage}[t]{0.49\textwidth}
      \centering
      \includegraphics[width=\linewidth]{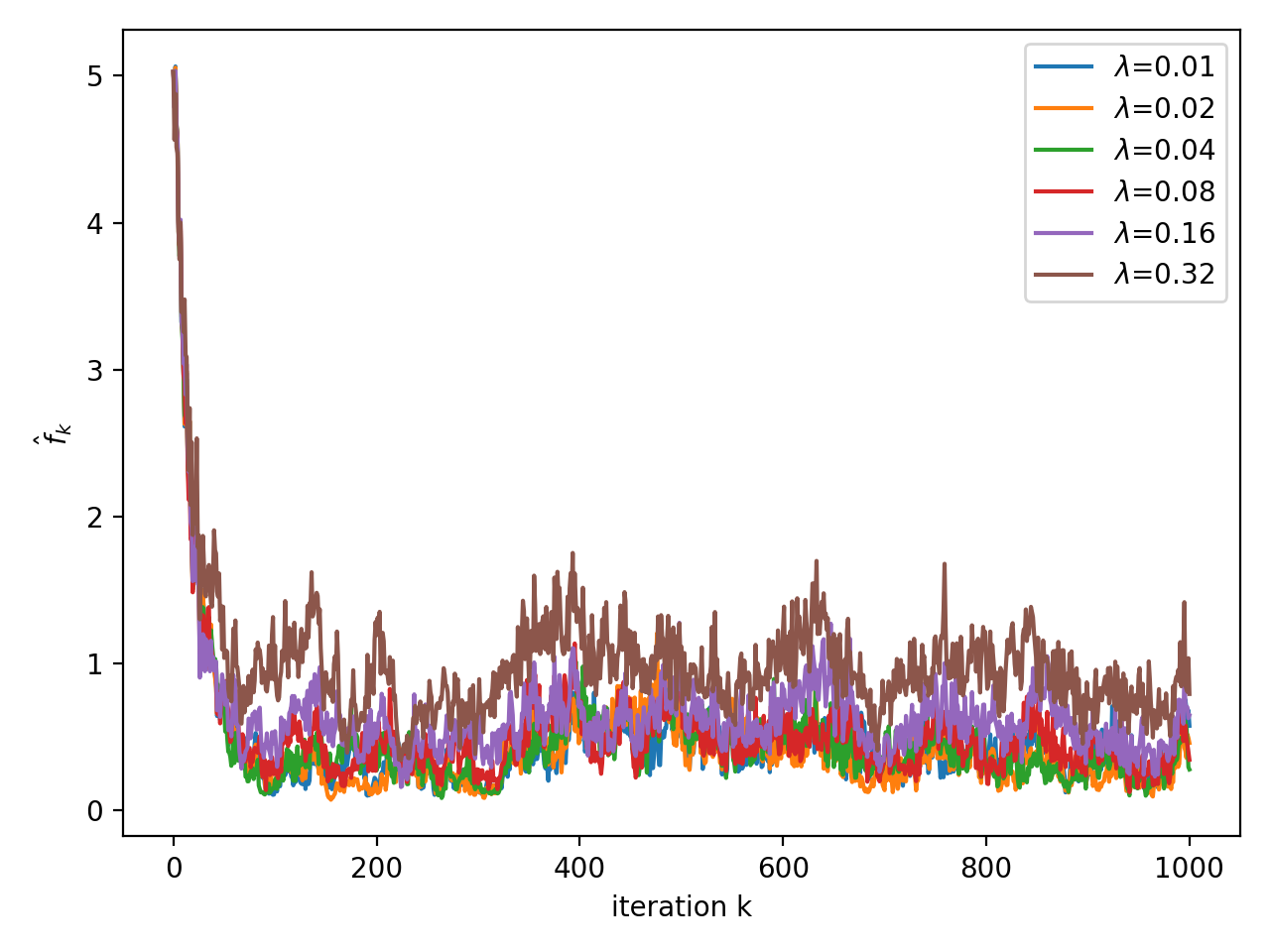}
    \end{minipage}\hfill
    \begin{minipage}[t]{0.49\textwidth}
      \centering
      \includegraphics[width=\linewidth]{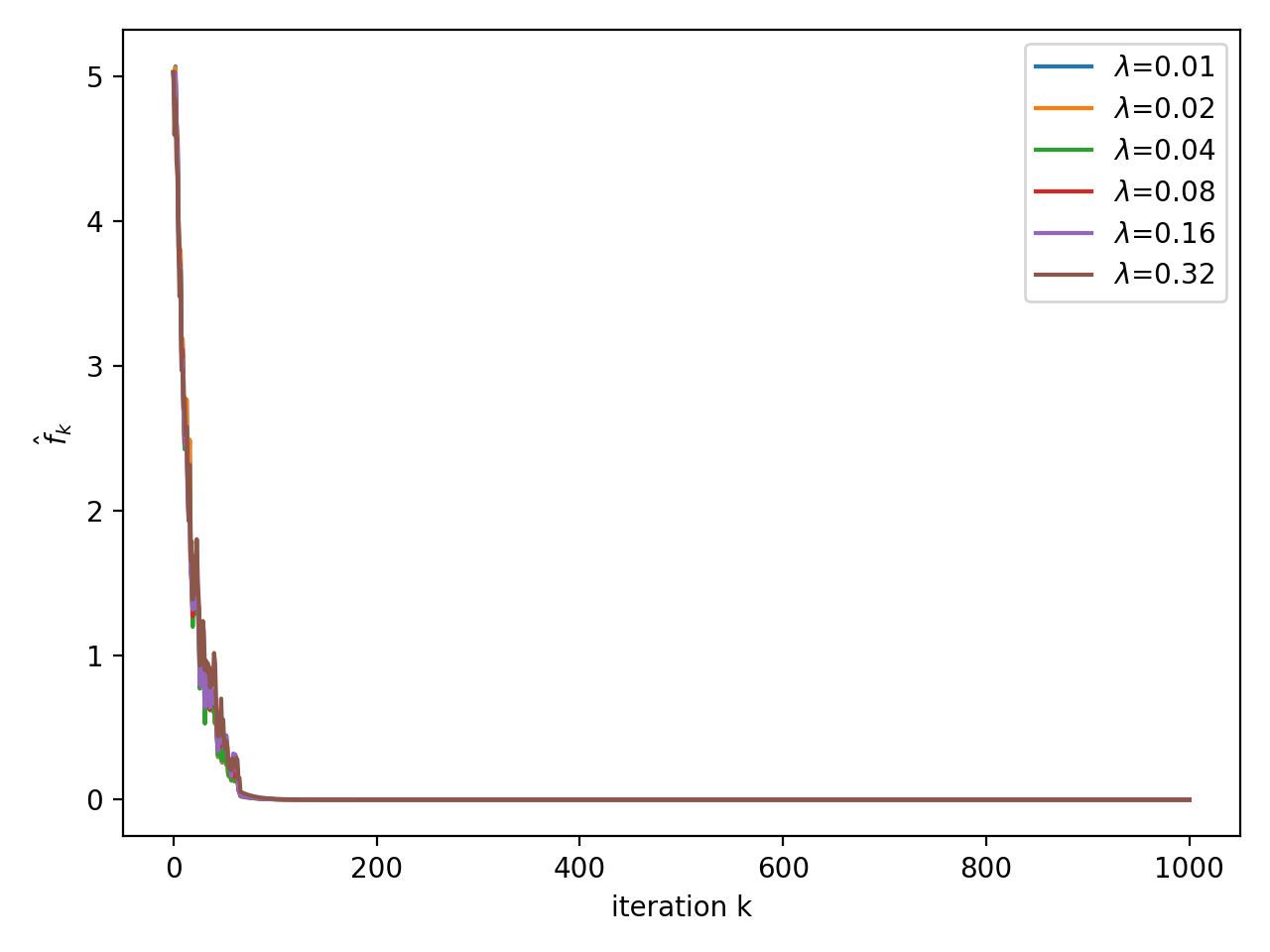}
    \end{minipage}%
  }
  \caption{\scriptsize 1D double-well: $\hat f_k$ (100-trajectory average) for multiple $\lambda$.
  Left: $\tau=0$. Right: $\tau=\tfrac12\sqrt{2u_{\max}}$.}
  \label{fig:1d-trunc-opt}
\end{figure}

\begin{exm}[2D Gaussian mixture, \cite{dong2021replicaexchangenonconvexoptimization}]
\label{exm:2d-gauss-mix}
Let  $\Omega=[-1,5]^2 $. 
Consider the Gaussian-mixture objective  \begin{equation} \label{eq:gauss_template}
f(\bm{x})
=
-\sum_{i=1}^M w_i \exp \left(
-\frac{1}{2}\left(\bm{x}-\bm{m}_i\right)^T \Sigma_i^{-1}\left(\bm{x}-\bm{m}_i\right)
\right),
\end{equation}  where $M=25$
and  $\{\bm m_i\}_{i=1}^{25}$ are placed on the grid $\{0,1,2,3,4\}^2$ (row-major order),
the covariances are $\Sigma_i=0.1\,I_2$, and
the mixture weights $\{w_i\}$ are listed in Table~\ref{tab:gm2d-mixture-wide}. In 
Figure~\ref{fig:gm2d-contour-surface}, we plot the function and observe multiple wells and saddle regions; the deepest well is associated with the largest weight, where $\bm{x}_{*}=(3,2)$.
\end{exm}

\begin{table}[t]
\centering
\caption{Component means $\bm m_i$ and weights $w_i$ for Example~\ref{exm:2d-gauss-mix} (row-major order on $\{0,1,2,3,4\}^2$).}

\label{tab:gm2d-mixture-wide}
\small

\scalebox{0.85}{
\setlength{\tabcolsep}{3pt}
\begin{tabular}{r c r @{\hspace{1em}} r c r @{\hspace{1em}} r c r @{\hspace{1em}} r c r @{\hspace{1em}} r c r}
\toprule
$i$ & $m_i$ & $w_i$ & $i$ & $m_i$ & $w_i$ & $i$ & $m_i$ & $w_i$ & $i$ & $m_i$ & $w_i$ & $i$ & $m_i$ & $w_i$ \\
\midrule
1  & (0,0) & 0.4559 & 2  & (0,1) & 0.2559 & 3  & (0,2) & 0.3089 & 4  & (0,3) & 0.2974 & 5  & (0,4) & 0.2947 \\
6  & (1,0) & 0.4972 & 7  & (1,1) & 0.5326 & 8  & (1,2) & 0.3268 & 9  & (1,3) & 0.4997 & 10 & (1,4) & 0.5220 \\
11 & (2,0) & 0.4020 & 12 & (2,1) & 0.3167 & 13 & (2,2) & 0.5011 & 14 & (2,3) & 0.3068 & 15 & (2,4) & 0.4747 \\
16 & (3,0) & 0.4392 & 17 & (3,1) & 0.5339 & 18 & (3,2) & \textbf{1.6552} & 19 & (3,3) & 0.4931 & 20 & (3,4) & 0.4037 \\
21 & (4,0) & 0.3124 & 22 & (4,1) & 0.2915 & 23 & (4,2) & 0.3972 & 24 & (4,3) & 0.4242 & 25 & (4,4) & 0.2974 \\
\bottomrule
\end{tabular}
}
\end{table}
In this example, we study the effects of the control parameters
$(u_{\min},u_{\max},\rho,\lambda)$ on the resulting Langevin dynamics. Specifically, we  
vary one of the control parameters at a time while freezing the rest. The goal is to identify which choices most affect convergence to the global minimizer to guide the choices in subsequent examples.

We set  $N_{\partial \Omega} = 128$ in  \eqref{eq:loss_infinite} and  obtain $\kappa=4.1$ from \eqref{eq:kappa-beta}.
In Algorithm~\ref{alg:mirror-boundary}, we use the step size $\eta=0.016$.

\begin{figure}[t]
  \centering
  \resizebox{0.9\textwidth}{!}{%
    \begin{minipage}[t]{0.49\textwidth}
      \centering
      \includegraphics[width=\linewidth]{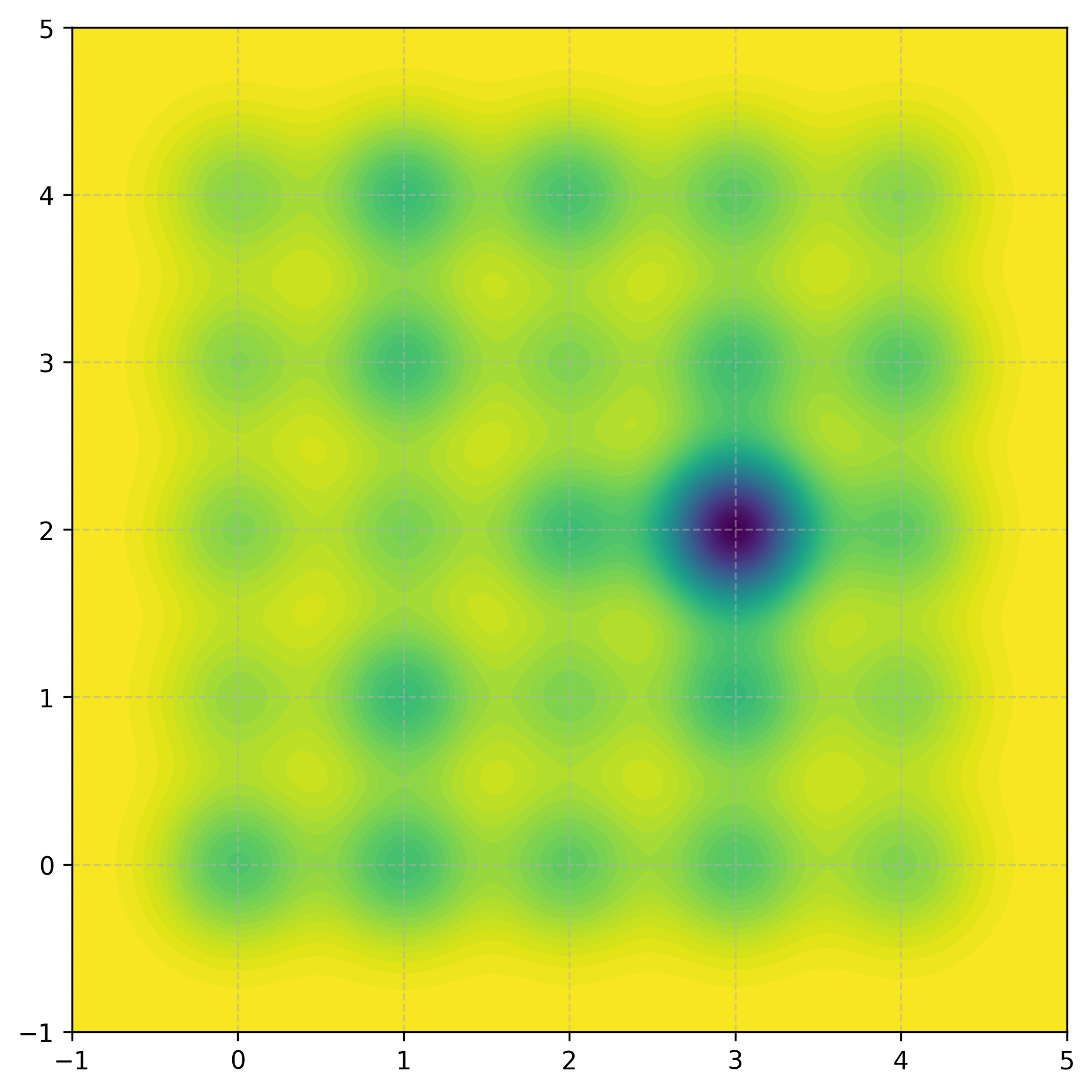}
    \end{minipage}\hfill
    \begin{minipage}[t]{0.49\textwidth}
      \centering
      \includegraphics[width=\linewidth]{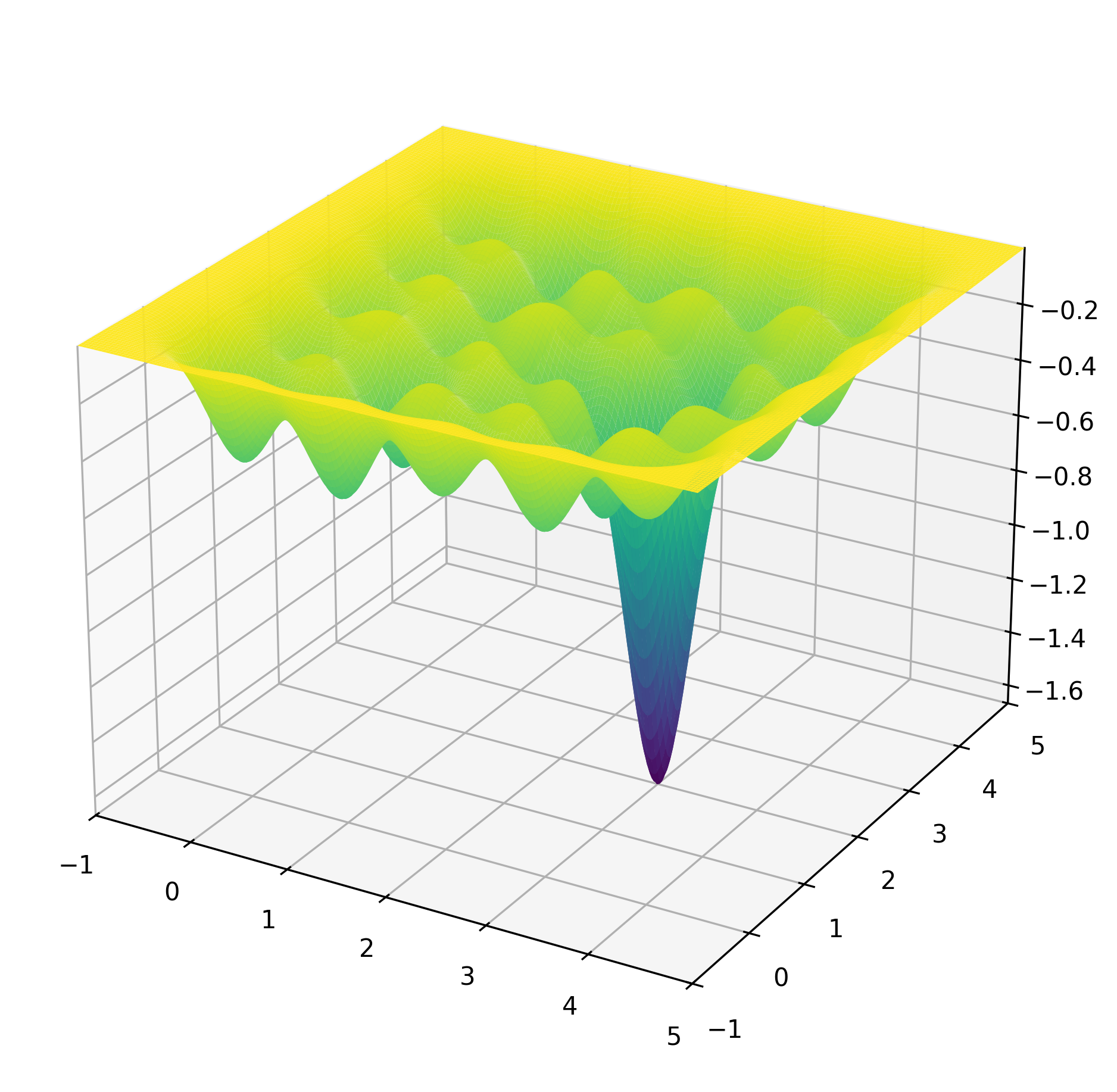}
    \end{minipage}%
  }
  \caption{\scriptsize Example~\ref{exm:2d-gauss-mix}: 2D Gaussian mixture on $\Omega=[-1,5]^2$ (contour and surface).
The landscape contains multiple wells; the deepest well is near $(3,2)$. Left: contour plot of $f$; right: surface plot.}
  \label{fig:gm2d-contour-surface}
\end{figure}

We first test the \textit{effects of the control bounds $u_{\min}$ and $u_{\max}$}.
In Figure~\ref{fig:exp-2d-umin-umax} (left), we plot the  error $\mathcal{E}_k$
(over $N_{\mathrm{traj}}=100$ trajectories) versus the Langevin iteration index $k$,
fixing $(\lambda,\rho,u_{\max},\tau)=(0.04,1.6,4\kappa,0)$ and varying
$u_{\min}=10^{-8},10^{-6},10^{-4},10^{-2}$.
In all tested cases, $\mathcal{E}_k\le 0.1$ by $k=500$ and remains below this level thereafter.
In Figure~\ref{fig:exp-2d-umin-umax} (right), we vary $u_{\max}= \kappa\times 2^{p} , p =-2,-1,0,1,2,3,4$,
while fixing $(u_{\min},\lambda,\rho,\tau)=(10^{-8},0.04,1.6,0)$. The numerical results 
show that a larger $u_{\max}$ leads to faster convergence. Specifically, for $u_{\max}\ge \kappa$, we have $\mathcal{E}_k\le 0.15$ by $k=500$ and continued improvement thereafter.
When $u_{\max}=\kappa/2$, we observe $\mathcal{E}_k\approx 0.399$ by $k = 500$ and about $0.132$ at the final iteration. 
When $u_{\max}=\kappa/4$, the error $\mathcal{E}_k$ remains large ($\mathcal{E}_k\approx 1.131$ by $k = 500$ and about $0.662$ at the end). This figure suggests that choosing $u_{\max}\geq  \kappa$ yields faster convergence while 
$u_{\min}$ should be kept close to zero.

\begin{figure}[t]
  \centering
  \resizebox{0.9\textwidth}{!}{%
    \begin{minipage}[t]{0.49\textwidth}
      \centering
      \includegraphics[width=\linewidth]{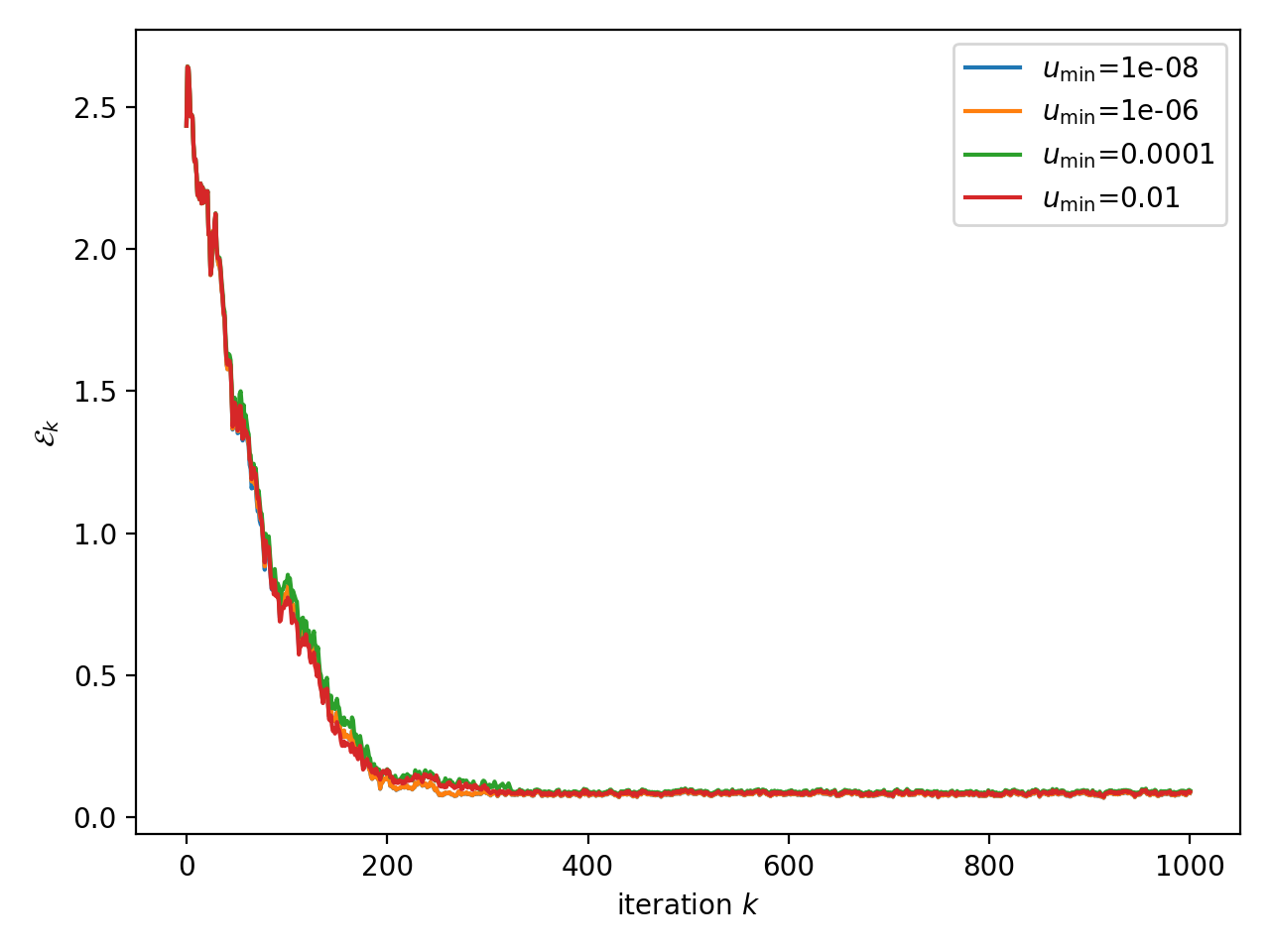}
    \end{minipage}\hfill
    \begin{minipage}[t]{0.49\textwidth}
      \centering
      \includegraphics[width=\linewidth]{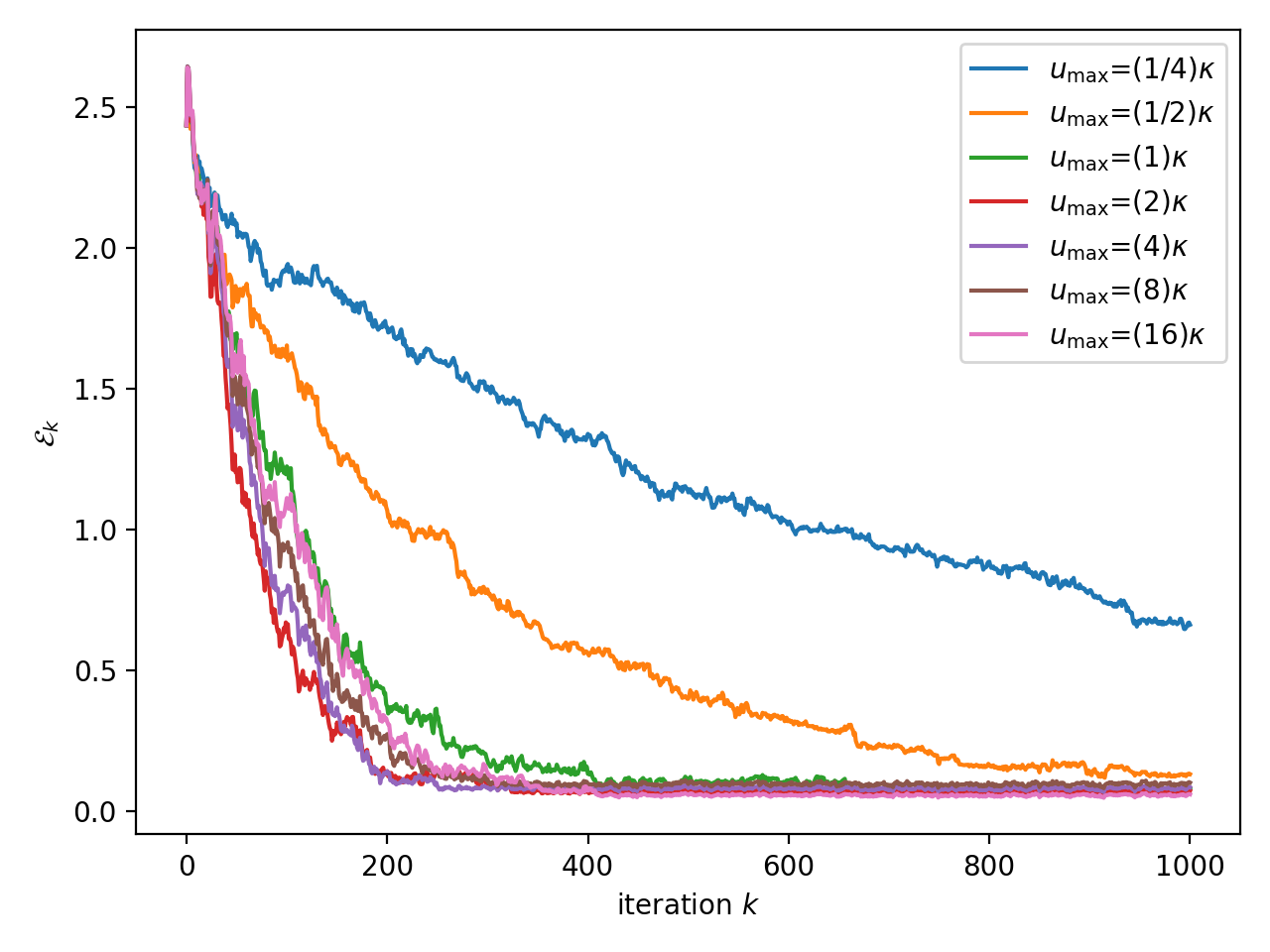}
    \end{minipage}%
  }
  \caption{\scriptsize Example~\ref{exm:2d-gauss-mix}: $\mathcal{E}_k$ (100-trajectory average). Left: varying $u_{\min}$ (curves nearly overlap).
Right: varying $u_{\max}$ (larger $u_{\max}$ yields faster convergence).}
  \label{fig:exp-2d-umin-umax}
\end{figure}

Next, we test the \textit{effects of $\lambda$ }. 
In Figure~\ref{fig:exp-2d-lambda-tau-rho} (left), we plot $\mathcal{E}_k$ versus the Langevin iteration index $k$, fixing $(u_{\min},\rho,u_{\max},\tau)=(10^{-8},1.6,4\kappa,0)$ and varying
$\lambda=0.01\times 2^p$, $p=0,1,\ldots,5$.
For $\lambda=0.01,0.02,0.04$, the trajectories enter the $0.1$-neighborhood of the global minima  by $k=500$ and remain in the region.
For $\lambda=0.08,0.16$, the error $\mathcal{E}_k$ is larger (about $0.1$--$0.2$ at $k = 500$), while for $\lambda=0.32$ the trajectories remain far from the minimizer set (with $\mathcal{E}_k$ oscillating around $1.8$).

Moreover, introducing truncation as in \eqref{eq:truncated-h}  improves the accuracy in the error $\mathcal{E}_k$, for large $\lambda$. Fixing $(u_{\min},u_{\max},\rho,\lambda)=(10^{-8},4\kappa,1.6,0.32)$ and take $\tau = s \sqrt{2u_{\max}},\; s = 0, \frac{1}{8}, \frac{1}{4}, \frac{1}{2}$. For $\tau = s \sqrt{2u_{\max}},\; s = \frac{1}{4}, \frac{1}{2}$, we observe that $\mathcal{E}_k<10^{-4}$ by $k=800$; see 
Figure~\ref{fig:exp-2d-lambda-tau-rho} (middle). 

Third, we vary the discount rate $\rho$. Specifically, we take $\rho = 0.2 \times 2^p, p = 1,2,...8$.
and fix $(u_{\min},\lambda,u_{\max},\tau)=(10^{-8},0.04,4\kappa,0)$. 
We observe from
Figure~\ref{fig:exp-2d-lambda-tau-rho} (right) that moderately large $\rho$'s lead to fast  convergence:
for $\rho=0.2, 0.4,0.8,1.6,3.2$,  we obtain $\mathcal{E}_k\le 0.15$ by $k=500$; whereas larger $\rho$ yields slow  convergence within a small number of iterations.

\begin{figure}[t]
  \centering
  \begin{minipage}[t]{0.32\textwidth}
    \centering
    \includegraphics[width=\textwidth]{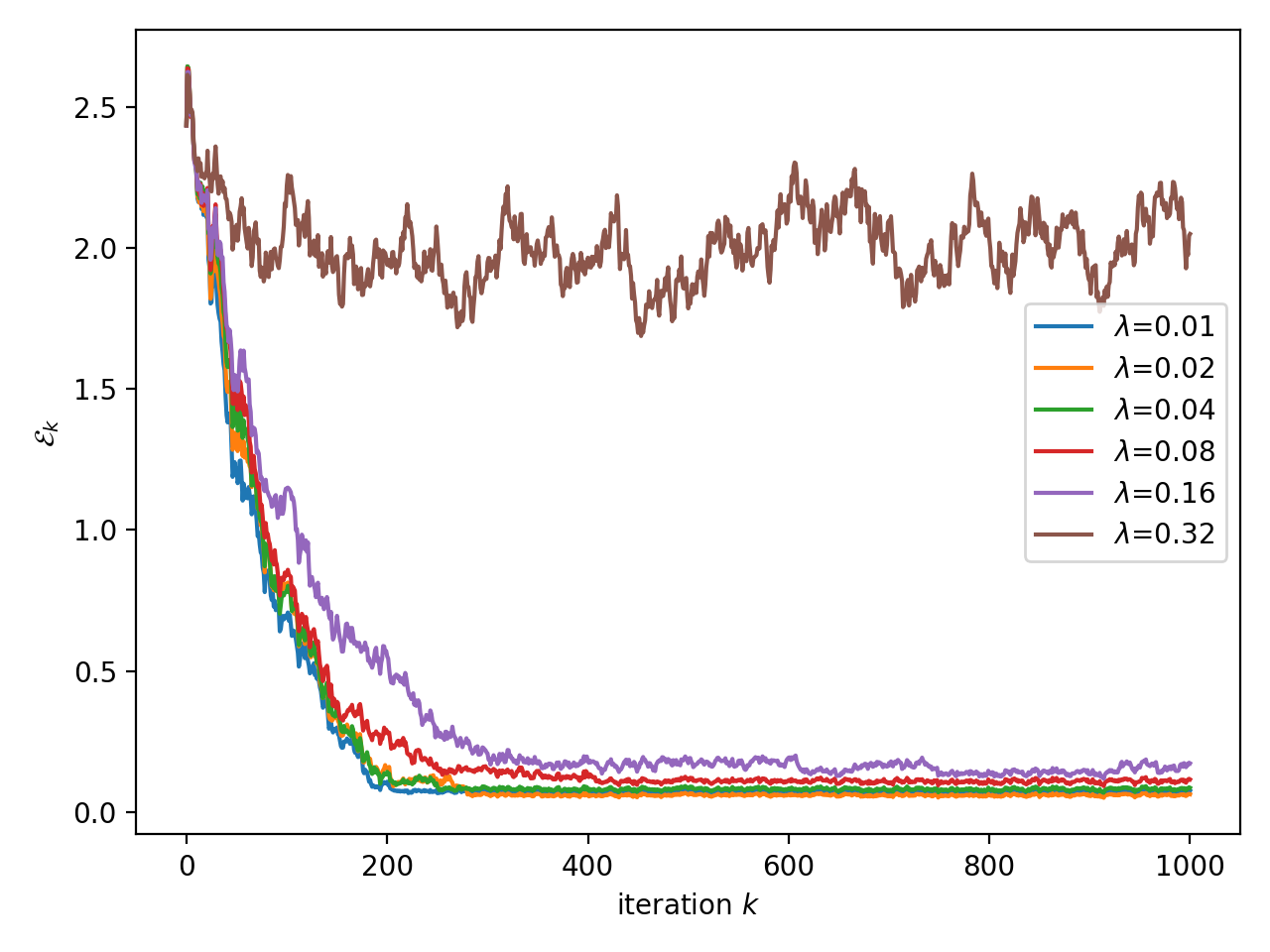}
  \end{minipage}\hfill
  \begin{minipage}[t]{0.32\textwidth}
    \centering
    \includegraphics[width=\textwidth]{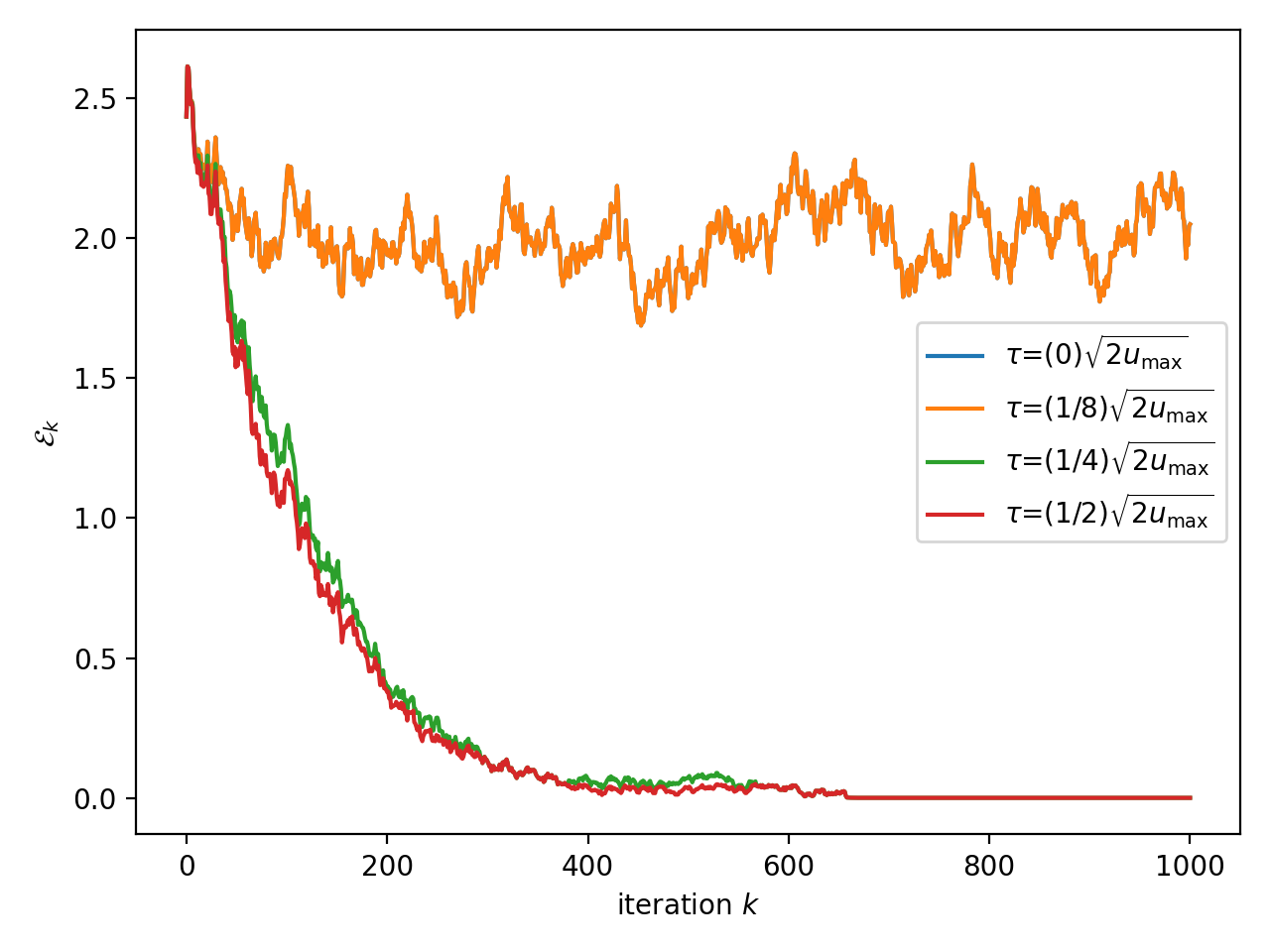}
  \end{minipage}\hfill
  \begin{minipage}[t]{0.32\textwidth}
    \centering
    \includegraphics[width=\textwidth]{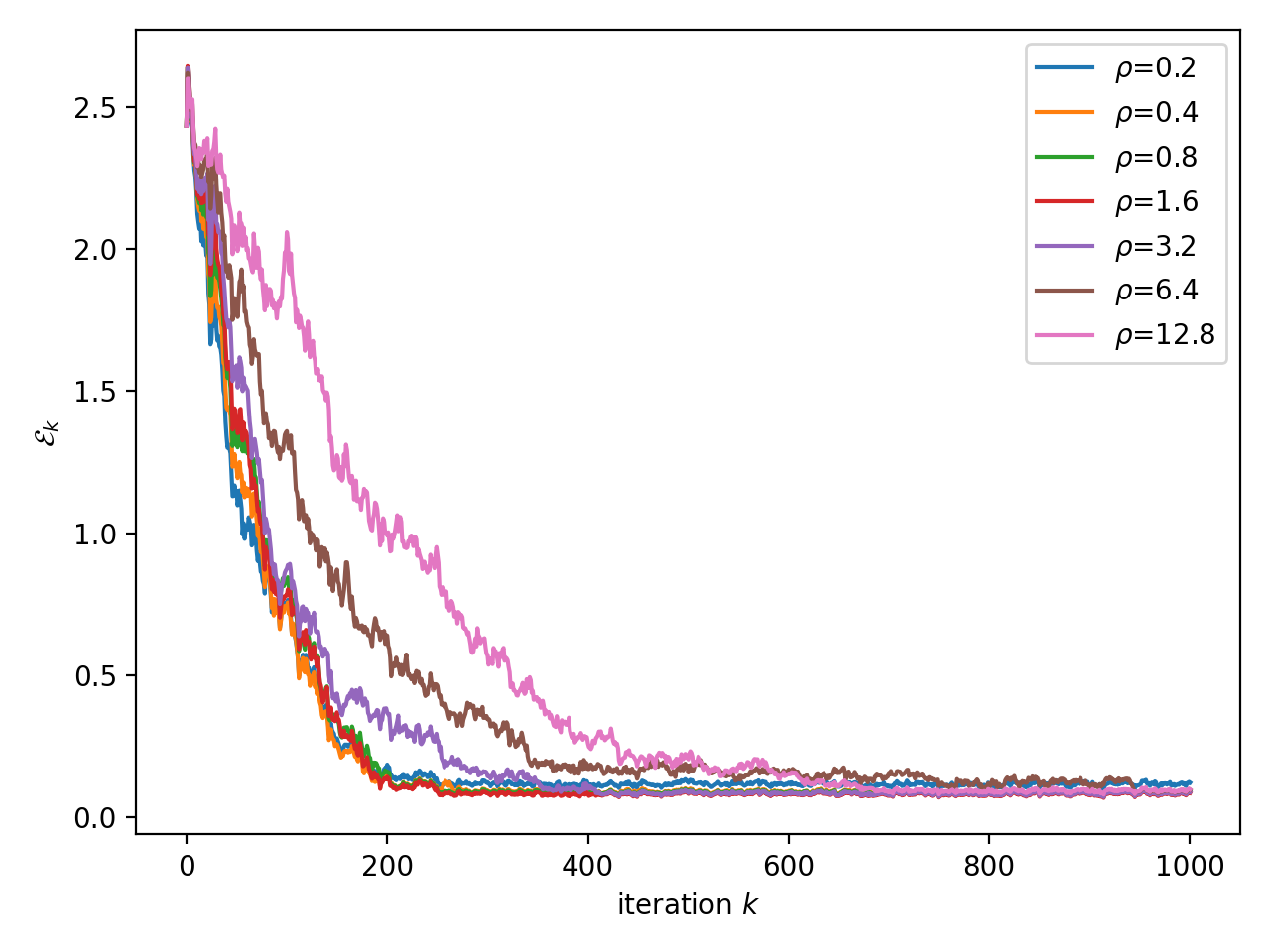}
  \end{minipage}

\caption{\scriptsize Example~\ref{exm:2d-gauss-mix}: $\mathcal{E}_k$ (100-trajectory average).
Left: varying $\lambda$. Middle: varying $\tau$ with $(u_{\max},u_{\min},\lambda)=(4\kappa,10^{-8},0.32)$.
Right: varying $\rho$. Smaller $\lambda$ and truncation for large $\lambda$ improve convergence; larger $\rho$ mainly slows early iterations.}

  \label{fig:exp-2d-lambda-tau-rho}
\end{figure}

In summary, the numerical results suggest that satisfactory performance of the stochastic Langevin dynamics for the minimization problem depends more on   $u_{\max}$ (no smaller than $\kappa$) and the regularization rate $\lambda$ than on 
$u_{\min}$ and the discount rate $\rho$.
Also, truncation in the noise coefficients  \eqref{eq:truncated-h} improves convergence for large $\lambda$ and both small and large $\rho$'s work well, but large $\rho$'s tend to slow down the convergence in early iterations.

In the following examples, we will follow the observations from Example \ref{exm:2d-gauss-mix} to set $(u_{\min},u_{\max},\rho,\lambda)$. Moreover, we only vary the truncation level $\tau$ to illustrate its practical role in stabilizing the trajectories.
We present additional simulations with various $\lambda$'s in Appendix~\ref{app:extra-sweeps}.


\begin{exm}[Easom function,  minima in a plateau]
\label{exm:easom}
Let  $\Omega=[-10,10]^2$ and 
\[
f(\bm{x})=-\cos(x_1)\cos(x_2)\exp\!\bigl(-(x_1-\pi)^2-(x_2-\pi)^2\bigr),
\qquad \bm{x}=(x_1,x_2)\in\Omega.
\]
The landscape is nearly flat around the unique global minimizer at $\bm{x}^*=(\pi,\pi)$.
\end{exm}

In this example, we choose $N_{\partial \Omega} = 4096$ in 
\eqref{eq:loss_infinite} and compute $\kappa=0.56$ in \eqref{eq:kappa-beta}.
In  Algorithm \ref{alg:mirror-boundary}, we use the step size $\eta=0.128$.  %
 We set $(u_{\min},u_{\max},\rho,\lambda)=(10^{-8}, 16 \kappa,0.2,0.01)$.

 We plot in
Figure~\ref{fig:bench-mark} (left) the errors of numerical global minimizers   $\mathcal{E}_k$ versus the Langevin iteration index $k$ for several truncation levels $\tau$. We observe 
that larger truncation improves terminal accuracy: $\tau=s\sqrt{2u_{\max}},\, s =\frac{1}{8}, \frac{1}{4}, \frac{1}{2}$ yields the smaller terminal $\mathcal{E}_k$ (near zero),
whereas smaller $\tau = 0, \frac{1}{16}\sqrt{2u_{\max}}$ give  errors $\mathcal{E}_k$'s. 
The numerical results in this example show that the proposed PINN-based temperature from the eHJB equations can help find a global minimizer of  non-convex functions with plateau landscapes.

\begin{exm}[Hartmann 6D function]
\label{exm:6d-hartmann}
Consider the   Hartmann function
\[
f(\bm{x})=\frac{1}{1.94}\Bigl(-\sum_{i=1}^4 \alpha_i
\exp\!\Bigl(-\sum_{j=1}^6 A_{ij}(x_j-P_{ij})^2\Bigr)+2.58\Bigr),
\quad \bm{x}\in\Omega=[0,1]^6,
\]
with $\alpha=(1.0,1.2,3.0,3.2)^\top$ and    $A,\,P\in\mathbb{R}^{4\times 6}$  given below: 
{\scriptsize
\[
\mathbf{A}=\begin{bmatrix}
10&3&17&3.50&1.7&8\\
0.05&10&17&0.1&8&14\\
3&3.5&1.7&10&17&8\\
17&8&0.05&10&0.1&14
\end{bmatrix},
\quad
\mathbf{P}=10^{-4}\begin{bmatrix}
1312&1696&5569&124&8283&5886\\
2329&4135&8307&3736&1004&9991\\
2348&1451&3522&2883&3047&6650\\
4047&8828&8732&5743&1091&381
\end{bmatrix}.
\]
}

This function is multimodal with multiple local minima;
the global minimizer is $\bm{x}_*=(0.20169,0.150011,0.476874,0.275332,0.311652,0.6573)$ and
$f(\bm{x}_*)=-3.32237$.
\end{exm}

We set $N_{\partial \Omega} = 4096$ in \eqref{eq:loss_infinite} and obtain $\kappa=16$ from  \eqref{eq:kappa-beta}. 
In  Algorithm \ref{alg:mirror-boundary}, we take the step size $\eta=0.016$. Set  $(u_{\min},u_{\max},\rho,\lambda)=(10^{-8}, \kappa,1.6,0.02)$.

We plot in Figure~\ref{fig:bench-mark} (right)  the errors of numerical global minimizers  $\mathcal{E}_k$ versus the Langevin iteration index $k$, for $\tau=s\sqrt{2u_{\max}}, s=0,\tfrac{1}{16},\tfrac{1}{8},\tfrac{1}{4},\tfrac{1}{2}$.
we observe that $\tau=\tfrac{1}{16}\sqrt{2u_{\max}}$ yields the best performance, while larger truncation levels increase the error $\mathcal{E}_k$ when $k$ is large.
For comparison, with $\tau=0$ we observe a plateau at $\mathcal{E}_k\approx 0.075$ for $k\ge 200$, whereas $s=\tfrac{1}{16}$ yields  smaller $\mathcal{E}_k$ at later iterations. The results suggest that a moderate truncation is preferred for this 6D multimodal benchmark.

\begin{figure}[t]
  \centering
  \resizebox{0.8\textwidth}{!}{%
    \begin{minipage}[t]{0.49\textwidth}
      \centering
      \includegraphics[width=\linewidth]{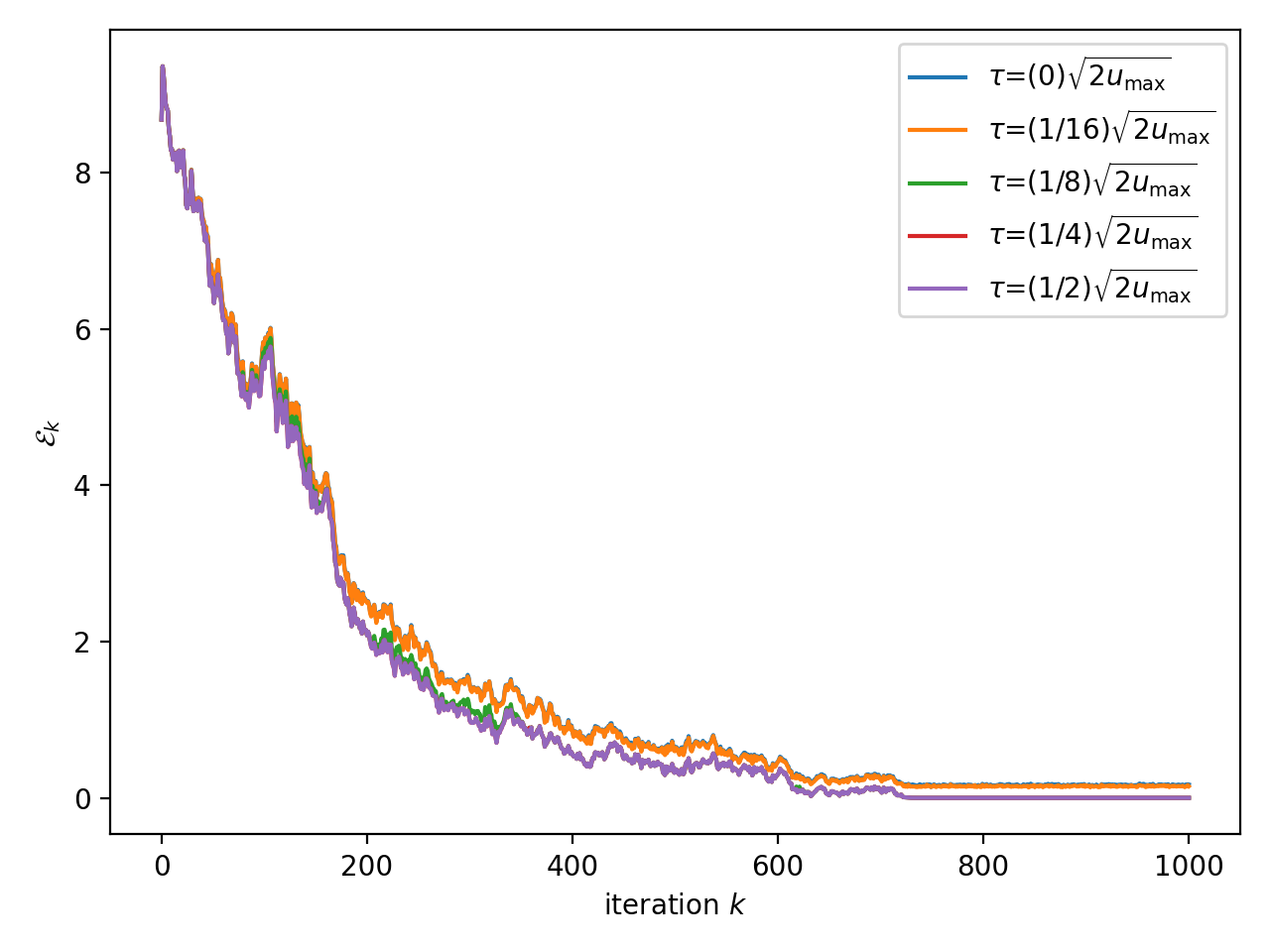}
    \end{minipage}\hfill
    \begin{minipage}[t]{0.49\textwidth}
      \centering
      \includegraphics[width=\linewidth]{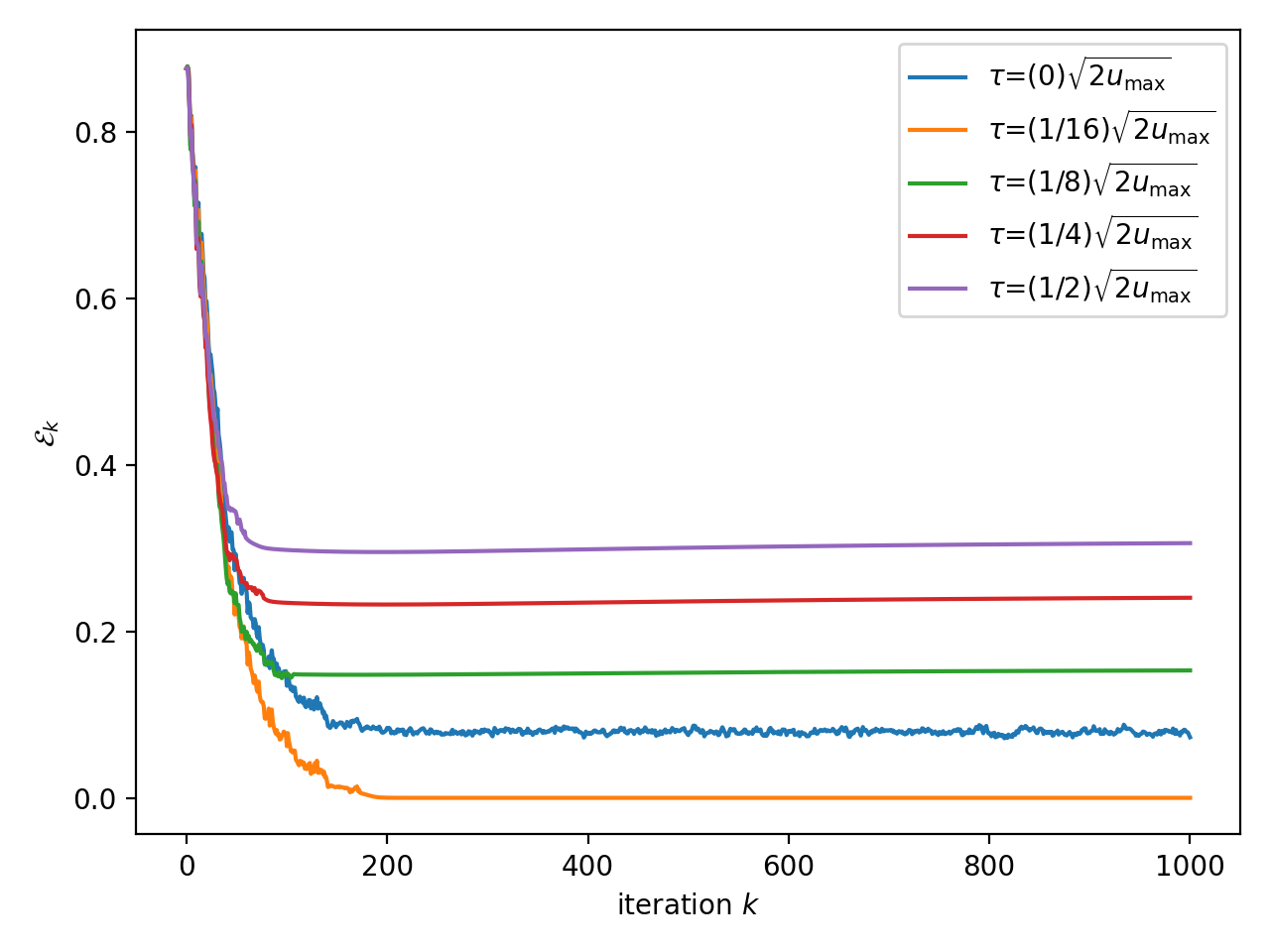}
    \end{minipage}%
  }
  \caption{\scriptsize 
Left: Example~\ref{exm:easom} (Easom), $\mathcal{E}_k$ for $\tau = \sqrt{2u_{\max}}, s=0,\tfrac{1}{16},\tfrac{1}{8},\tfrac{1}{4},\tfrac{1}{2}$.
Right: Example~\ref{exm:6d-hartmann} (Hartmann-6), $\mathcal{E}_k$ for the same truncation set.}
  \label{fig:bench-mark}
\end{figure}



\section{Conclusion and discussion}
\label{sec:conclusion}
The minimization problems solved with Langenvin dynamics were studied via controlled diffusions and the corresponding exploratory HJB equation. 
We apply physics-informed neural networks to solve the exploratory HJB equation. 
The quantity of most interest, 
the trace of the Hessian, is proved to converge with order one in the regularization parameter $\lambda$. 
Further analysis shows that the 
error induced by the physics-informed neural networks is at the order of the  magnitude of PDE residual (Theorem \ref{thm:convergence-lambda-pertub}). 
We demonstrate through several representative examples that the approach can be effective for solving minimization problems up to 6 dimensions, including overcoming saddle points and finding a global minimum on a plateau.
However, the computational cost is high due to the  PINN solvers for eHJB equations.  Further computational cost reduction may come from using the solution structure of eHJB equations as well as more advanced techniques, such as tensor neural networks (e.g. \cite{wang2024tensorneuralnetworkshighdimensional}) and random mini-batch in dimensionality \cite{Hu_2024}.  
Also, systematic studies on classical optimization benchmark functions should be conducted. 


\section*{Acknowledgment}
The authors would like to thank Professors 
Zuo Quan Xu of Hong Kong Polytechnic University, 
 George Yin of  
University of Connecticut,
and
Xun Yu Zhou of Columbia University for valuable discussions.


\appendix

\section{Monotone finite differences and Howard policy iteration for the classical HJB equation}
\label{app:fd-howard}

We briefly summarize the numerical scheme used to compute the PDE-based reference
solution of the classical HJB equation~\eqref{eq:classical-hjb} 
on $\Omega=(x_L,x_R)$
with homogeneous Neumann boundary conditions
\begin{equation}
v_x(x_L)=0,\qquad v_x(x_R)=0.
\label{eq:neumann_app}
\end{equation}
Denote $b(\bm{x}):=-f'(\bm{x})$ for simplicity. 
We use a uniform grid $x_i=x_L+i\Delta x$ for $i=0,\dots,N-1$ with
$\Delta x=(x_R-x_L)/(N-1)$ and unknowns $v_i\approx v(x_i)$.
The Neumann conditions~\eqref{eq:neumann_app} are enforced by
\begin{equation}
\frac{v_1-v_0}{\Delta x}=0,\qquad \frac{v_{N-1}-v_{N-2}}{\Delta x}=0.
\label{eq:neumann_fd_app}
\end{equation}
At interior nodes $i=1,\dots,N-2$, we discretize $v_{xx}$ by a central difference and
$b\,v_x$ by an upwind difference:
\[
v_{xx}(x_i)\approx \frac{v_{i+1}-2v_i+v_{i-1}}{(\Delta x)^2},\qquad
b(x_i)v_x(x_i)\approx
b_i^+\frac{v_i-v_{i-1}}{\Delta x}+b_i^-\frac{v_{i+1}-v_i}{\Delta x},
\]
where $b_i=b(x_i)$, $b_i^+=\max(b_i,0)$, and $b_i^-=\min(b_i,0)$.
For a fixed grid control $u_i\in[u_{\min},u_{\max}]$, the resulting interior equation is
\begin{equation}
-\rho v_i + f_i
+ u_i\frac{v_{i+1}-2v_i+v_{i-1}}{(\Delta x)^2}
+ b_i^+\frac{v_i-v_{i-1}}{\Delta x}
+ b_i^-\frac{v_{i+1}-v_i}{\Delta x}
=0,\, i=1,\dots,N-2.
\label{eq:fd_row_app}
\end{equation}
We solve the resulting nonlinear system using Howard's policy iteration.

In Example  \ref{exm:double-well}, 
$\Omega=[x_L,x_R]=[-6,6]$ and we use a uniform grid with $N=1001$ points (spacing $\Delta x=(x_R-x_L)/(N-1)$), discount rate $\rho=0.4$, and diffusion control $u(\bm{x})\in[u_{\min},u_{\max}]=[0.2,142]$; Howard iteration is run in double precision and stops when $\|u^{(k+1)}-u^{(k)}\|_\infty<\varepsilon_u=10^{-4}$ or when $k$ reaches $k_{\max}=2000$.

\section{Additional tests for Easom and Hartmann-6}\label{app:extra-sweeps}

We present in  Figure ~\ref{fig:app-lambda-sweeps}  additional tests with several $\lambda$'s for the Easom and Hartmann-6 optimization problems under fixed
$(u_{\min},u_{\max},\rho,\tau)$ and step size $\eta$. We observe that 
small errors $\mathcal{E}_k$, $k\leq 1000$, of the numerical global minimizers for $\lambda=0.01,0.02,0.04,0.08,0.16$.  However, a relatively small $\lambda$ is preferred in simulations. 

\begin{figure}[t]
  \centering
  \begin{minipage}[t]{0.49\textwidth}
    \centering
    \includegraphics[width=\linewidth]{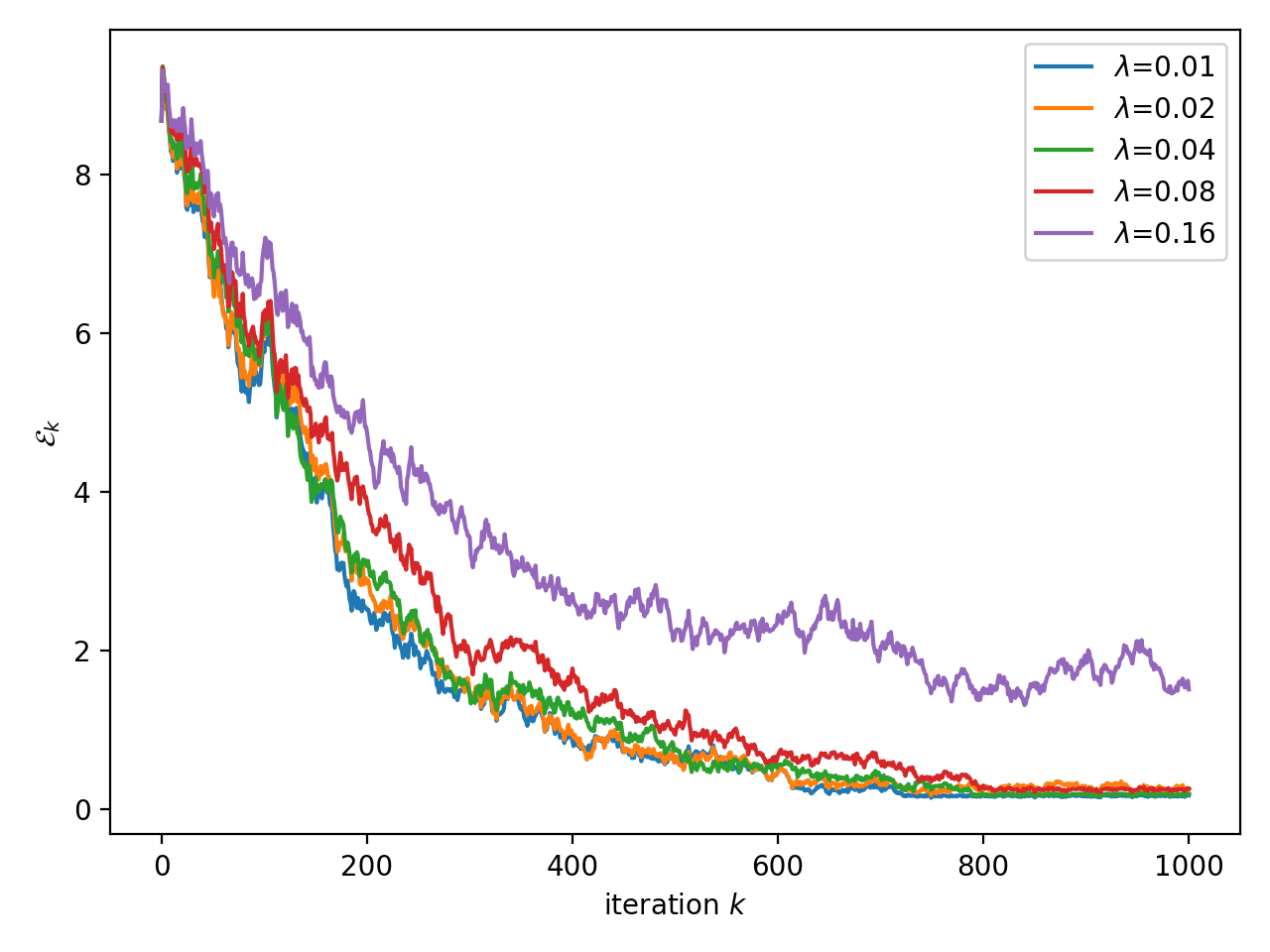}
    {\scriptsize (a) Easom: varying $\lambda$}
  \end{minipage}\hfill
  \begin{minipage}[t]{0.49\textwidth}
    \centering
    \includegraphics[width=\linewidth]{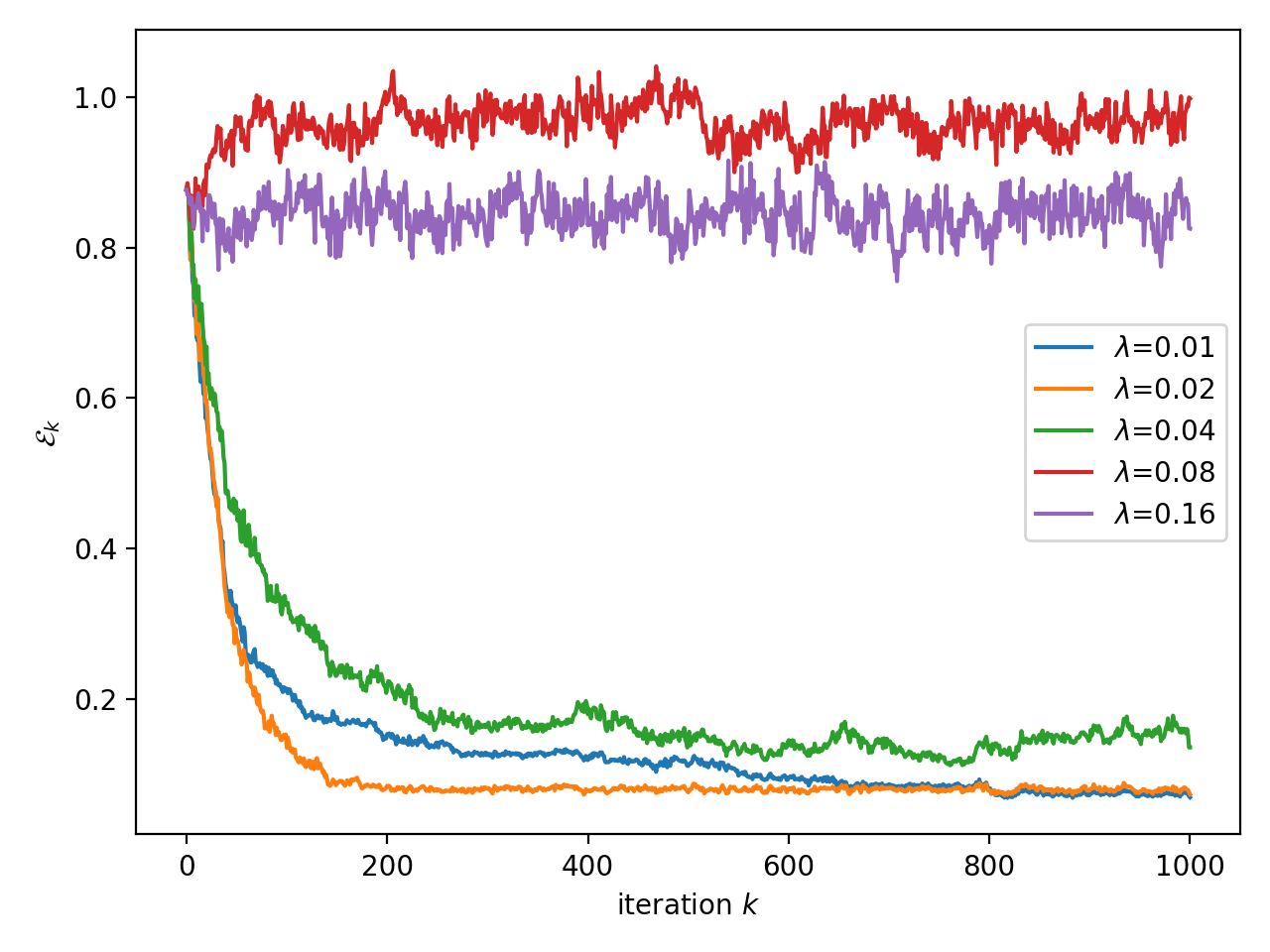}
    {\scriptsize (b) Hartmann-6: varying $\lambda$}
  \end{minipage}

  \caption{\scriptsize Additional  tests of the errors $\mathcal{E}_k$ of numerical global minimizers with various $\lambda$'s. 
  (a) Easom with $(u_{\min},u_{\max},\rho,\tau)=(10^{-8},16\kappa,0.02,0)$ and step size $\eta=0.128$.
  (b) Hartmann-6 with $(u_{\min},u_{\max},\rho,\tau)=(10^{-8},\kappa,1.6,0.02)$ and step size $\eta=0.016$.
  In both cases, $\lambda=0.01,0.02,0.04,0.08,0.16$.}
  \label{fig:app-lambda-sweeps}
\end{figure}

 

\section{Proofs}
From Theorem \ref{thm:convergence-lambda-pertub}, $\Omega$ is open and bounded with a smooth boundary. 
Then we obtain that  $v,v_\lambda\in C^{2,\alpha}$ with  
$\alpha \in (0,1)$ depends on dimension $d$ and   $u_{\min}>0$.
%
Define $
    H_\lambda(\bm{x}) =-\lambda \ln \int_{u_{\min}}^{u_{\max}} \exp\left(-  {\lambda}^{-1} {\tr X}  u \right) du$, and 
    $
    H(\bm{x})=(u_{\min}\chi_{\tr X>0}+u_{\max}\chi_{\tr X<0})\tr X$, where 
    $\chi_{A}(\cdot)$ is the indicator function on the set $A$. 

By direct calculations (or see Lemma 4.4 of  \cite{TangZhangZhou2022}),    we have, 
for any $X,Y \in \Real^{d\times d}$,  
    \begin{equation}\label{eq:h-lipschitz}
    u_{\min} \|X-Y\| \leq \abs{H (\bm{x})-H(Y)}\leq u_{\max} \|X-Y\|.
    \end{equation}

\begin{lem}\label{lem:estimate-g-lambda}
Suppose that $X\in C^{0,\alpha}$, $\alpha\in (0,1)$. There exists
$C>0$ independent  of $\lambda$ such that 
%
\[  \|H_\lambda(\bm{x})-H(\bm{x})\|_{C^{0,\gamma}} \leq C \lambda^{1-\gamma/\alpha},  \, 
0\leq \gamma<\alpha.
\] 
\end{lem}

\begin{proof}
%
Let $Y= X(y)$ and $X=X(\bm{x})$, $x,y\in \Omega$.
By direct calculations,  
$\norm{DH_\lambda}\leq C$ and 
$\norm{D^2 H_\lambda}\leq C\lambda^{-1}$, where $C$ is independent of $\lambda$. 
The conclusion for $\gamma=0$ then follows from the Lipschitz continuity of $H$ and $H_\lambda$ and direct calculations. %

By Taylor's formula on $H_\lambda(\bm{x})$ around $Y$, we have 
\begin{align*}
   & \abs{H_\lambda(\bm{x}) -H(\bm{x}) - \big({H_\lambda(Y) -H(Y)}\big)}\\
   &\leq
    \abs{H(\bm{x}) -H(Y)}+ \abs{H_\lambda(\bm{x}) -H_\lambda(Y)}\\
    &     \leq C\norm{X-Y} + \abs{DH_\lambda:(X-Y)} + 
     C\norm{D^2H_\lambda}\norm{X-Y}^2 \\
     &\leq C\norm{X-Y}  + C \lambda^{-1}\norm{X-Y}^2,
     \end{align*}
     where the constant $C$ may change from step to step, while remaining a constant (the same applies to the rest of the argument). Since both $X$ and $Y$ are in $C^{0,\alpha}$, we have, for 
     $0<\gamma<\alpha$,
     \begin{equation*}
          \frac{\abs{H_\lambda(\bm{x}) -H(\bm{x}) - \big({H_\lambda(Y) -H(Y)}\big)}}{\abs{x-y}^\gamma}\leq 
     C \abs{x-y}^{\alpha-\gamma} +
     C \lambda^{-1}\abs{x-y}^{2\alpha-\gamma}. 
     \end{equation*}
It can be readily verified that when 
$\abs{x-y}^\alpha $ is at the order of $ \lambda$, the right-hand side will reach its supremum, which is at the order of 
$\lambda^{1-\gamma/\alpha}$.
\end{proof}

\begin{thm}\label{thm:convergence-order-lambda}
Suppose that 
$f\in C^2(\Omega)$. Then  
   \begin{eqnarray}
\sqrt{u_{\min}} \norm{\nabla v_\lambda-\nabla v}_{L^\infty(\Omega)} +  
 u_{\min}\norm{\Delta v_\lambda-\Delta v}_{L^\infty(\Omega)}\leq 
 C   \big( \norm {w_\lambda}_{L^\infty(\Omega)}  
   + \norm{g_\lambda}_{L^\infty(\Omega)}
   \big)\leq C\lambda .
  \label{eq:convergence-derivatives-infty}\end{eqnarray}
\end{thm}

\begin{proof}[Proof of Theorem \ref{thm:convergence-order-lambda}]
The difference $w_\lambda=v_\lambda- v$ satisfies the following equation
\begin{align}  \label{eq:difference-eqn}
\rho w_\lambda -\nabla f\cdot \nabla w_\lambda
 +   A_{\lambda} \Delta w_\lambda 
& =H(D^2v)- H_\lambda(D^2v) =: g_\lambda (\bm{x})  \end{align}
where we apply Taylor's theorem
$H_\lambda(D^2v_\lambda) - H_\lambda(D^2v)=A_\lambda (\Delta v_\lambda - \Delta v)
$ and 
$A_\lambda=\int_0^1
D H_\lambda \left(
D^2v + s(D^2v_\lambda - D^2v)
\right) 
\,ds$.

Since $ u_{\min} I \leq A_\lambda$,
the equation of $w_\lambda$ is uniformly elliptic. 
By the equation \eqref{eq:difference-eqn} and 
$f\in C^{2}$, we have 
\begin{equation}
   u_{\min}\norm{ \Delta w_\lambda}_{L^{\infty}(\Omega)}
   \leq C_K\big( \norm {w_\lambda}_{L^\infty(\Omega)} + \norm{\nabla w_\lambda}_{L^\infty(\Omega)}
   + \norm{g_\lambda}_{L^\infty(\Omega)}
   \big),
\end{equation}
and by the interpolation inequality
$ \norm{ \nabla w_\lambda}_{L^\infty (\Omega)} \leq 
    C_K \norm{w_\lambda}_{L^\infty (\Omega)}^{1/2}  \norm{ \Delta w_\lambda}_{L^\infty (\Omega)}^{1/2}$, 
we obtain that  
\begin{equation*} u_{\min}\norm{\Delta w_\lambda}_{L^\infty(\Omega)}\leq 
 C   \big( \norm {w_\lambda}_{L^\infty(\Omega)}  
   + \norm{g_\lambda}_{L^\infty(\Omega)}
   \big).
\end{equation*} 
Then by Lemma \ref{lem:estimate-g-lambda} and $\norm {w_\lambda}_{L^\infty(\Omega)}\leq C\lambda $ (\cite{TangZhangZhou2022}), 
${u_{\min}} \norm{\Delta w_\lambda}_{L^\infty(\Omega)}\leq C\lambda$. 
Then, by the interpolation inequality above,  
$
\sqrt{u_{\min}} \norm{\nabla w_\lambda}_{L^\infty(\Omega)}\leq C \lambda 
$ and thus 
\eqref{eq:convergence-derivatives-infty} holds.
\end{proof}

\begin{thm}\label{thm:convergence-order-pertub}
Suppose that $f\in C^{2}(\Omega)$ and 
$\widetilde{R}_\lambda(\bm{x},\phi)\in L^\infty(\Omega)$ in \eqref{eq:exploratory-hjb-elliptic-pertubed}. Then we have 
   \begin{align*}
   \sqrt{u_{\min}}\norm{ \nabla (v_\lambda-v_{\lambda,\phi})}_{L^\infty(\Omega)}
   + u_{\min}\norm{ \Delta (v_\lambda-v_{\lambda,\phi})}_{L^\infty(\Omega)} &\leq C \varepsilon \norm{\widetilde{R}_\lambda(\cdot,\phi)}_{L^\infty},  \, \varepsilon>0.
\end{align*} 
\end{thm}
The proof is similar to that of Theorem 
\ref{thm:convergence-order-pertub} and thus omitted. 
The only difference lies in the fact that \eqref{eq:exploratory-hjb-elliptic-pertubed} is not an actual equation and is rather a value of the 
eHJB operator evaluated at 
$v_{\lambda,\phi}$. 
 


%
\end{document}